\documentclass[10pt, final]{article}

\usepackage[T1]{fontenc}		     
\usepackage[utf8]{inputenc}			 
\usepackage[english]{babel}

\usepackage{amsmath}
\usepackage{amssymb}
\usepackage{amsthm}
	\theoremstyle{plain}

		\newtheorem{thm}{Theorem}[section]	

		\newtheorem{cor}[thm]{Corollary}		
		\newtheorem{lem}[thm]{Lemma}		
		\newtheorem{prop}[thm]{Proposition}

	\theoremstyle{definition}
		\newtheorem{defn}[thm]{Definition}	
		\newtheorem{ex}[thm]{Example}		

	\theoremstyle{remark}
		\newtheorem{rem}[thm]{Remark}		
		\newtheorem{note}[thm]{Notation}		

\numberwithin{equation}{section}	

\usepackage{mathtools}		
\usepackage{mathrsfs}		
\usepackage{eucal}			

\usepackage{braket}			
\usepackage[all,pdf]{xy}		
\usepackage{mparhack}		
\usepackage{relsize}			

\usepackage{a4wide}		

\usepackage{booktabs}		
\usepackage{multirow}		
\usepackage{caption}		
\usepackage{rotating}
\usepackage{subfig}

\usepackage{varioref}		
\usepackage{footmisc}

\usepackage{graphicx}		
\usepackage{epsfig}

\usepackage{enumerate}		
\usepackage[colorlinks=true, linkcolor=black, citecolor=black, 
urlcolor=blue]{hyperref}		
\mathtoolsset{showonlyrefs}
\newcommand{\GL}{\mathrm{GL}}
\newcommand{\trasp}[1]{{#1}^\mathsf{T}}					%
		
\newcommand{\R}{\mathbb{R}}	
\newcommand{\C}{\mathbb{C}}		

\newcommand{\Sp}{\mathrm{Sp}}

\newcommand{\Lagr}{\Lambda}

\newcommand{\iRS}{\iota^{\textup{RS}}}		
\newcommand{\iAr}{\iota^{\textup{Ar}}}		
\newcommand{\iL}{\iota^{\textup{L}}}

\newcommand{\Lin}{\mathscr{L}}
\newcommand{\Linsa}{\mathscr{L}^{\textup{sa}}}
\newcommand{\Graph}{\mathrm{Gr\,}}
\newcommand{\mul}{\mathrm{mul\,}}
\newcommand{\ssp}{\mathfrak{sp}}

\newcommand{\Sym}{\mathrm{Sym}}

\newcommand{\Bsym}{\mathrm{B}_{\textup{sym}}}

\newcommand{\codim}{\mathrm{codim}\,}

\newcommand{\Bil}{\mathscr{B}}					

\newcommand{\norm}[1]{\left\| #1 \right\|}			
				
\newcommand{\iMor}{\mathrm n_-}		
\newcommand{\iMorse}[1]{\iota_{#1}}		
\newcommand{\coiMor}{\coindex}	
\newcommand{\PT}[1]{\mathcal P_T({#1})}

\renewcommand{\L}{L}	
\newcommand{\N}{\mathbb{N}}		

\newcommand{\iMas}[1]{\iota_{#1}}

\newcommand{\iCLM}{\iota^{\scriptscriptstyle{\mathrm{CLM}}}}

\newcommand{\iCZ}{\iota^{\scriptscriptstyle{\mathrm{CZ}}}}
\newcommand{\itriple}{\iota}
\newcommand{\Z}{\mathbb{Z}}		

\newcommand{\irel}{I}

\newcommand{\coindex}{\mathrm{n}_+\,}
\newcommand{\iindex}{\mathrm{n}_-\,}
\newcommand{\nullity}{\mathrm{n}_0\,}
\newcommand{\noo}[1]{\overset {\mbox{%
\lower1pt\hbox{${\scriptstyle o}$}}}{\mathrm n}_{\mbox{%
\lower2pt\hbox{$\scriptstyle#1$}}}}

\DeclareMathOperator{\diag}{diag}		
\DeclareMathOperator{\spfl}{sf}			
\DeclareMathOperator{\sgn}{sgn}		

\renewcommand{\leq}{\leqslant}
\renewcommand{\geq}{\geqslant}

\renewcommand{\=}{\coloneqq}			

\newcommand{\email}[1]{\href{mailto:#1}{\textsf{#1}}}

\usepackage{bm}

\newcommand{\Id}{\mathrm{Id}}

\title{Sturm theory with applications in geometry \\ 
and \\ classical mechanics}
\author{Vivina L. Barutello, Daniel Offin, Alessandro Portaluri\thanks{The first and third 
authors are partially supported by Prin 2015 ``Variational methods, with applications to problems in mathematical physics and geometry” No.~$\mathrm{2015KB9WPT\_001}$.},  Li Wu}
\date{\today}

\date{\today}
\begin{document}
 \maketitle

\begin{abstract}

Classical Sturm non-oscillation and  comparison theorems as well as the Sturm theorem on zeros for solutions of second order differential equations have a natural symplectic version, since they describe the rotation of a line in the phase plane of the equation. In the higher dimensional  symplectic version of these theorems, lines are replaced  by Lagrangian subspaces and intersections with a given line are replaced by non-transversality instants with a distinguished Lagrangian subspace.  Thus the symplectic Sturm theorems describe some properties of the Maslov index. 

Starting from the celebrated  paper of Arnol'd on symplectic Sturm theory for optical Hamiltonians,  we provide a generalization of his results to general Hamiltonians.  We finally apply  these results  for detecting some geometrical information about the distribution of conjugate and focal points on semi-Riemannian manifolds and for studying the geometrical properties of the solutions space of singular Lagrangian systems arising  in Celestial Mechanics. 

\vskip0.2truecm
\noindent
\textbf{AMS Subject Classification:\/}  53D12, 70F05, 70F10.\vskip0.1truecm
\noindent
{\bf Keywords\/}: Maslov index, Conley-Zehnder index, Hamiltonian dynamics, conjugate points, Kepler problem.
\end{abstract}




\section*{Introduction}\label{sec:Introduction and  description of the problem}

Symplectic Sturm theory has a lot of predecessor, like Morse, Lidskii, Bott, Edwards, Givental  who proved the Lagrangian nonoscillation of the Picard-Fuchs equation for hyperelliptic integrals. The classical Sturm theorems on oscillation, non-oscillation, alternation and comparison for a second-order ordinary differential equation have a symplectic nature. They, in fact, describe 
the rotation of a straight line through the origin of the phase plane of the equation. A line through the origin is a special 1-dimensional subspace of the phase plane: it is, in fact a Lagrangian subspace. 

Starting from this observation, as clearly observed and described by Arnol'd in \cite{Arn86}, the higher-dimensional symplectic generalization of the Sturm theory has been obtained by replacing lines by Lagrangian subspaces and instants of intersections between lines, by  instants of non-transversality. Such instants in the terminology of Arnol'd  has been termed {\em moments of verticality\/}. Thus, in higher dimension, the rotation of a straight line through the origin has been replaced by the evolution of a Lagrangian subspace through the phase flow of the linear Hamiltonian system in the phase space. The  phase flow  defines, in this way,  a curve of Lagrangian subspaces and moments of verticality correspond to the intersection instants between this curve and a hypersurface (with singularities) in the Lagrangian Grassmannian manifold, called (in the Arnol'd terminology), the {\em train\/} of a distinguished Lagrangian subspace. Such a train  is a transversally oriented variety and by using such an orientation, it is possible to define an integer-valued intersection index, called Maslov index. Generically, in a local chart of the Lagrangian Grassmannian manifold, the contribution to the Maslov index of a $\mathscr C^1$-Lagrangian curve, is through the signature of a quadratic form, the so-called {\em crossing form\/}. In some particular cases it can actually happen that the signature coincides with the coindex, namely with the positive inertia index of the crossing form. If this property holds at each crossing, the Lagrangian curve is called {\em positive curve\/} or {\em plus curve\/}. This property strongly depends upon the choice of the distinguished Lagrangian subspace in the sense that a curve could be a plus curve with respect to a Lagrangian subspace $L_0$ but not respect  to a different $L_1$. Often, in the applications, such a distinguished  Lagrangian subspace $L_0$ is uniquely determined by the boundary conditions imposed on the problem. 

As already observed, Arnol'd proved Sturm nonoscillation, alternation and comparison theorems in the case of {\em optical\/} or {\em positively twisted\/} Hamiltonians, namely Hamiltonian functions such that the flow lines of the lifted Hamiltonian flow on the Lagrangian Grassmannian manifold level are positive curves with respect to a distinguished Lagrangian. This kind of Hamiltonians often occur in applications. It is well-known, in fact,  that Legendre convex Hamiltonians in the cotangent bundle with the canonical symplectic form are optical with respect to the vertical section.

However, in many interesting applications, the lifted Hamiltonian flow at the Lagrangian Grassmannian level is not a positive curve with respect to a fixed Lagrangian subspace or, otherwise said, the Hamiltonians are not optical with respect to a  distinguished section of the cotangent bundle. This is for instance the case of the evolution of a Lagrangian subspace through the phase flow induced by the Jacobi deviation equation  along a spacelike geodesic in a Lorentzian manifold or more generally of a geodesic of any causal character on a semi-Riemannian manifold having non-trivial signature. (Cfr. \cite{PPT04, MPP05, MPP07, GPP04} and references therein).

Surprisingly, Sturm alternation and  comparison theorems  still hold in the case of not optical Hamiltonians. Sturm alternation theorem actually gives an estimate between the difference of the Maslov indices computed with respect to two different Lagrangians. By using the transition functions of the atlas of the Lagrangian Grassmannian, authors in \cite[Proposition 3.3 \& Corollary 3.4]{JP09} proved an estimate on the difference of Maslov indices and then applied this estimate  in order to obtained some comparison results between conjugate and focal points in the semi-Riemannian world.  Inspired by \cite{JP09}, in this paper we prove, among others, a sharper estimate of the difference between two Maslov indices with respect to two different Lagrangian subspaces (see Theorem \ref{thm:main-corollary-alternation}). The main idea in order to provide this estimates mainly relies on the H\"ormander index whose vocation is precisely to measure such a difference. 

By using this estimate together with the Bott-Long  type iteration formula we provide, in Proposition~\ref{thm:stima-conley-iterate}, an estimate between the Conley-Zehnder index of an iterated periodic orbit of a Hamiltonian system and the Conley-Zehnder index of the orbit on its prime period.
Furthermore we give an estimate between the Conley-Zehnder index and the Maslov index with respect to a distinguished Lagrangian $L$ of a Lagrangian curve constructed by letting evolving $L$ under the phase flow of a Hamiltonian system. The interest for this study is mainly related to the following  fact. In the case of  symmetric periodic orbits it is possible to associate in a natural fashion the Conley-Zehnder index as well as the Maslov index with respect to a fixed Lagrangian subspace.   In the case of autonomous Hamiltonian systems with discrete symmetries (e.g. reversible Hamiltonian systems) the (symmetric) periodic solutions can be interpreted either as periodic orbits or as Lagrangian intersection points and hence they have both indices naturally associated. 

Another interesting result of the present paper is Theorem \ref{thm:sturm-nonoscillation} which is nothing but the Sturm nonoscillation theorem. This result is somehow hybrid and has in its own the Lagrangian and the Hamiltonian nature of the problem. If the Hamiltonian is natural (meaning that it is the sum of the kinetic and the potential energy) in which the kinetic part is a positive quadratic form in the momentum variables  and the potential part is a non-positive definite quadratic form in the configuration variables, then the moment of non-transversality between the Lagrangian curve induced by the lifted phase flow at Lagrangian Grassmannian level and the Dirichlet  Lagrangian is less or equal than the number of degrees of freedom.  We observe that these  assumptions on the kinetic and potential energy, don't insure that the induced Lagrangian curve is a plus curve with respect to any Lagrangian subspace different from the Dirichlet (which is the  Lagrangian corresponding to the coordinate plane of vanishing configuration variables). However, these signature assumptions, insure that the Lagrangian function is non-negative. This is a pretty important information and gives deep insight on the spectral analytic properties of the problem. In fact, up to a shifting constant (discussed in Section \ref{sec:variational}) that is bounded by the number of degrees of freedom, the Maslov index coincides with the Morse index. Now, under the signature assumptions on the kinetic and potential energy, it follows that the Morse index is zero and hence the the Maslov index is bounded by the number of degrees of freedom. This, however, is not the end of the story, since the bound on the Maslov index doesn't imply, in general, a bound on the total number of crossing  instants. However, in the case of plus curve, it does. This is why in the theorem the Maslov intersection index is considered with respect to the Dirichlet Lagrangian (and in fact such a  Hamiltonian is Dirichlet optical, being Legendre convex). 

An extremely useful result in applications is Theorem \ref{thm:comparison-1}: a  generalized version of the Sturm comparison theorem. In this case, on the contrary, is not important to work with plus Lagrangian curves. This fact, has been already recognized by the third author in \cite{MPP07}. Loosely speaking, the monotonicity between Hamiltonian vector fields implies an inequality on the Maslov index and if the Hamiltonian system is induced by a second order  Lagrangian system  $\mathscr C^2$-convex in the velocity, this implies an inequality on the Morse indices.  From a technical viewpoint the proof of this result is essentially based upon the homotopy invariance of the Maslov index. An essential ingredient in the proof is provided by a spectral flow formula for paths of unbounded self-adjoint firs order (Fredholm) operators with dense domain in $L^2$. 

Finally in the last section we provide some applications essentially in differential topology and classical mechanics. More precisely, we prove some interesting new estimates about the conjugate and focal points along geodesics on semi-Riemannian manifolds, improving the estimates provided by authors in \cite[Section 4]{JP09}. We stress on the fact that classical comparison theorems for conjugate and focal points in Riemannian manifolds and more generally on Lorenzian manifolds but for timelike geodesics, requires curvature assumptions or Morse index arguments. On general semi-Riemannian manifolds having non-trivial signature, the curvature is never bounded and the index form has always infinite Morse index and co-index.  The second application we provide is based upon an application of the Sturm comparison theorem to  the  Kepler problem in the plane with fixed (negative) energy.

Considerable effort has been focused on improving the readability of the manuscript and on explaining  the main ideas and involved techniques.

\tableofcontents


\subsection*{Notation}

For the sake of the reader, we introduce some notation that we shall use henceforth without further reference throughout the paper. 
\begin{itemize}
\item[-] We denote by $V, W$  finite dimensional real vector spaces; $\Lin(V,W)$ and $\Bil(V,W)$ respectively the  vector spaces of all {\em linear operators\/} $T: V \to W$ and of {\em bilinear forms\/} $B: V\times W \to \R$; by $V^*$ we denote the dual space of $V$, i.e. $V^*=\Lin(V, \R)$. In  shorthand notation we set 
$\Lin(V)\=\Lin (V,V)$ and $\Bil(V)\=\Bil(V,V)$. $\Linsa(V)$ denotes the subset of $\Lin(V)$ of all linear self-adjoint operators on $V$. 
There is a {\em canonical isomorphism\/} 
\[
 \Lin(V, W^*) \ni T \to B_T \in \Bil(V,W) \textrm{ such that  }  B_T(v,w)\= T(v)(w), \quad \forall\, v \in V,\,
\forall\, w \in W.
\]
$\Id_V$ or in shorthand notation just $\Id$ denotes the identity;
\item[-] For $T \in \Lin (V, W)$, we define the {\em pull-back\/} of $C \in \Bil(W)$ through the map  $T$ as 
\[
 T^*: \Bil(W) \to \Bil(V) \textrm{ given by  } T^*(C)\= C(T\cdot, T\cdot)
\]
and if $T$ is an {\em isomorphism\/} we define the {\em push-forward\/} of $B \in \Bil(V)$  through $T$ as the map:
\[
 T_*: \Bil(V) \to \Bil(W)  \textrm{ given by  } T_*(B)\= B(T^{-1}\cdot, T^{-1}\cdot).
\]
Given  a linear operator $T:V \to V$, we denote by $\Graph(T)\subset V^2$ its graph. If $T=\Id$, its graph coincide with the diagonal subspace $\Delta\subset V\times V$.
\item[-] $\Bsym(V)$ is   the vector space of all symmetric bilinear forms on $V$. For any $B\in\Bsym(V)$, we denote by $\iindex(B)$,
$\nullity(B)$ and $\coindex(B)$ respectively speaking the {\em index},  the   {\em nullity} and the {\em coindex\/}  of $B$.
The {\em  signature\/} of  $B$ is the difference
$\sgn(B)\=\coindex(B)-\iMor(B)$\\ $B$ is termed  {\em
non-degenerate\/} if $\nullity(B)=0$. 
\item[-]  $(V, \omega)$ denotes a $2n$-dimensional (real) symplectic vector space and  $J$ denotes a complex structure on $V$;  $\Sp(V, \omega)$ the {\em symplectic group \/};  $\ssp(V,\omega)$ denotes the symplectic Lie algebra. $\GL(V)$ denotes  the general linear group. The  {\em symplectic group\/}  of $(\R^{2n}, \omega)$ is denoted by $\Sp(2n)$ and its  Lie algebra simply  by $\ssp(2n)$. We refer to a  matrix in $\ssp(2n)$ as the set of  {\em Hamiltonian matrices\/}. 
\item[-] $
 \PT{V,\omega}\=\Set{\psi\in \mathscr C^0\big([0,T], \Sp(V,\omega)\big)| \psi(0)=\Id} $
 where $\PT{V,\omega}$  is equipped with the topology induced from $(V, \omega)$. $\PT{2n}$ denotes the set $\PT{V,\omega}$ in the case in which $(V, \omega)= (T^*\R^n,\omega_0)$.
 \item[-] $\Lambda(V,\omega)$ denotes the Lagrangian Grassmannian of $(V,\omega)$ whereas $\Lambda(n)$ denotes the Lagrangian Grassmannian of the standard $2n$-dimensional symplectic space. 
 \end{itemize}


\subsection*{Acknowledgements}
The third name author wishes to thank all faculties and staff at the Queen's University (Kingston)  for providing excellent working conditions during his stay and especially his wife, Annalisa, that  has been extremely supportive of him throughout this entire period  and has made countless sacrifices to help him  get to this point.

%

\section{Variational framework and an Index Theorem}\label{sec:variational}

This section is devoted to recall some basic definitions and results about the Lagrangian and Hamiltonian dynamics that we shall need later on.  The main result in this section is a Morse-type index theorem given at Theorem \ref{thm:index} relating the Morse index of a critical point $x$ of the Lagrangian action functional with the Maslov-type index of $z_x$ corresponding to $x$ through the Legendre transform. Our basic references are \cite{Dui76, APS08, HS09}.

Let $T\R^n\cong \R^n \oplus \R^n$ be the tangent space of $\R^n$ endowed with coordinates $(q,v)$. Given  $T>0$ and the Lagrangian function  $L \in \mathscr C^2([0,T]\times T\R^n, \R)$, we assume that the following two assumptions hold
\begin{itemize}
\item[{\bf (L1)\/}]  $L$ is $\mathscr C^2$-convex with respect to $v$,  that is the quadratic form  
\[
\langle \partial_{vv} L(t,q,v) \cdot, \cdot \rangle \quad \textrm{ is positive definite } \quad \forall\,  t\in [0,T],\ \   \forall\, (q,v)\in T\R^n
\]
\item[{\bf (L2)\/}] $L$ is {\em exactly quadratic\/} in the velocities $v$ meaning  that the function $L(t,q,v)$ is a polynomial of degree at most $2$ with respect to $v$.
 \end{itemize}
Under the assumption (L1)  the Legendre transform  defined by 
\begin{equation}
\mathscr L_L:[0,T] \times T\R^n \to [0,T] \times T^*\R^n, \qquad (t,q,v) \mapsto \big(t,q,D_v L(t,q,v) \big)
\end{equation}
is a $\mathscr C^1$ (local) diffeomorphism. 
\begin{rem}
	The assumption (L2) is in order to guarantee that the action functional is twice Frechét differentiable. It is well-known, in fact, that the smoothness assumption on the Lagrangian is in general not enough. The growth condition required in (L2) is related to the regularity of the Nemitski operators. For further details we refer to \cite{PWY19}  and references therein. 
\end{rem}
We denote by $H\=W^{1,2}([0,T], \R^n)$ be the  space of paths having Sobolev regularity $W^{1,2}$ and we define the Lagrangian action functional $A: H\to \R$ as follows
\begin{equation}\label{eq:action}
	A(x)=\int_0^T L\big(t, x(t), x'(t)\big)\, dt.
\end{equation}
Let $Z \subset \R^n \oplus \R^n$ be a linear subspace and let us consider the linear subspace 
\[
H_Z\=\Set{x \in H| \big(x(0), x(T)\big) \in Z}.
\]
\begin{note}
In what follows we shall denote by $A_Z$ the restriction of the action $A$ onto $H_Z$; thus in symbols we have $A_Z\= A\big\vert_{H_Z}$.
\end{note}
It is well-know that critical points of the functional $A$ on $H_Z$ are weak (in the Sobolev sense) solutions of the following boundary value problem
\begin{equation}\label{eq:bvp}
\begin{cases}
\dfrac{d}{dt}\partial_v L\big(t, x(t), x'(t)\big)= \partial_q L\big(t, x(t), x'(t)\big), \qquad t \in [0,T]\\
\big(x(0), x(T)\big) \in Z, \quad \Big(\partial_v L\big(0, x(0), x'(0)\big), -\partial_v L\big(T, x(T), x'(T)\big)\Big)\in Z^\perp
\end{cases}
\end{equation}
where $Z^\perp$ denotes the orthogonal complement of $Z$ in $T^*\R^n$  
and up to standard elliptic regularity arguments, classical (i.e. smooth) solutions. 
\begin{rem}
We observe, in fact, that there is an identification of $Z\times Z^\perp$ and the  conormal subspace of $Z$, namely $N^*(Z)$ in $T^*\R^n$. For further details, we refer the interested reader to \cite{APS08}.	
\end{rem}
We assume that $x \in H_Z$ is a classical solution of the boundary value problem given in Equation~\eqref{eq:bvp}.  We observe that,  by assumption (L2) the functional $A$ is twice Fréchet differentiable on $H$. Being the evaluation map from  $H_Z$ into $H$ a smooth submersion, also the restriction $A_Z$ is twice  Fréchet differentiable and by this  we get that $d^2A_Z(x)$ coincides with $D^2A_Z(x)$.

By computing the second variation of $A_Z$ at $x$ we get
\begin{equation}
d^2A_Z(x)[\xi, \eta] =\int_0^T \big[\langle P(t)\xi'+Q(t) \xi, \eta'\rangle+ \langle \trasp{Q}(t)\xi', \eta \rangle + \langle R(t) \xi, \eta\rangle \big] \, dt, \qquad \forall\, \xi,\eta \in H_Z
\end{equation}
\begin{multline}
\textrm{ where }\  P(t)\=\partial_{vv}L\big(t,x(t),x'(t)\big), \quad  Q(t)\=\partial_{qv}L\big(t,x(t),x'(t)\big)\\ \textrm{ and finally }  R(t)\=\partial_{qq}L\big(t,x(t),x'(t)\big).
\end{multline}
Now, by linearizing the ODE  given in Equation \eqref{eq:bvp} at $x$, we finally get  the (linear) Morse-Sturm boundary value problem defined as follows
\begin{equation}\label{eq:MS-system}
\begin{cases}
	-\dfrac{d}{dt}\big[P(t)u'+ Q(t) u\big] + \trasp{Q}(t) u'+R(t)u=0, \qquad t \in [0,T]\\
	\big(u(0), u(T)\big) \in Z, \quad \Big(Pu'(0)+Q(0)u(0),-\big[P(T)u'(T)+Q(T) u(T)\big]\Big)\in Z^\perp
	\end{cases}
\end{equation}
We observe that $u$ is a weak (in the Sobolev sense)  solution of the boundary value problem given in  Equation \eqref{eq:MS-system} if and only if $u \in \ker I$. Moreover,  by elliptic bootstrap it follows that $u$ is a smooth solution.  

Let us now consider the {\em standard symplectic space\/} $T^*\R^n\cong \R^n \oplus \R^n$ endowed with the {\em canonical symplectic form\/}
\begin{equation}\label{eq:standard-form}
	\omega_0\big((p_1,q_1),(p_2,q_2)\big)\= \langle p_1, q_2 \rangle - \langle q_1, p_2 \rangle. 
\end{equation}
Denoting by $J_0$ the {\em (standard) complex structure\/}  namely the automorphism $J_0:T^*\R^n\to T^*\R^n$  defined  by  $  J_0(p,q)=(-q, p) $   whose associated matrix is  given by 
\begin{equation}\label{eq:J0}
J_0= \begin{pmatrix}
  0 & -\Id\\
  \Id & 0  
 \end{pmatrix} 
\end{equation}
it immediately follows that $\omega_0 (z_1,z_2)\= \langle J_0 z_1, z_2\rangle$ for all $z_1,z_2 \in T^*\R^n $.
\begin{note}
In what follows, $T^*\R^n$ is endowed with a coordinate system  $z=(p,q)$, where $p=(p_1, \dots, p_n)\in \R^n$ and $q=(q_1, \dots, q_n)\in \R^n$. we shall refer to $q$ as {\em configuration variables\/} and to $p$ as the {\em momentum variables\/}. 
\end{note}
By setting $z(t)\= \trasp{\big(P(t)u'(t)+Q(t)u(t), u(t)\big)}$,  the Morse-Sturm equation reduces  to the following (first order) Hamiltonian system in the standard symplectic space
\begin{multline}\label{eq:hamsys-finale}
z'(t)=J_0 B(t)\, z(t), \qquad t \in [0,T] \quad \textrm{ where }\\
B(t)\=\begin{bmatrix}
	P^{-1}(t) & -P^{-1}(t) Q(t)\\
	-\trasp{Q}(t)P^{-1}(t) & \trasp{Q}(t) P^{-1}(t) Q(t) -R(t)
\end{bmatrix}
\end{multline}
We now define the {\em double standard symplectic space\/} $(\R^{2n}\oplus \R^{2n}, -\omega_0 \oplus \omega_0)$ and we introduce the matrix  $\widetilde J_0\=\diag(-J_0,J_0)$ where $\diag(*,*)$ denotes the $2 \times 2$ diagonal block matrix. In this way,  the subspace $L_Z$ given  by 
\begin{equation}\label{eq:utile-p-focali}
L_Z\= \widetilde J_0(Z^\perp\oplus Z)
\end{equation}
is thus Lagrangian.  
\begin{note}
The following notation will be used throughout the paper. 
If $x$ is a solution of \eqref{eq:bvp} we denote by  $z_x$ the corresponding function defined by 
\begin{equation}\label{eq:z-x}
\big(t, z_x(t)\big)= \mathscr L_L\big(t,x(t),x'(t)\big).
\end{equation} 
\end{note}
\begin{defn}\label{def:I_Morse}
Let $x$ be a critical point of $A$. We denote by $\iMorse{ Z}(x)$ the {\em Morse index\/} of $x$ namely 
\[
\iMorse{Z}(x)\=\sup\Set{\dim L| L \subseteq H_Z \textrm{ and } d^2 A(x)_{L\times L} \textrm{ is negative definite}} \in \N \cup \{+\infty\}.
\]
Let $z_x$ be defined in  Equation \eqref{eq:z-x}. We define the {\em Maslov index of $z_x$\/} as the integer given by 
\begin{equation}
	\iMas{L_Z}(z_x)\=\iCLM\big(L_Z, \Graph \psi(t); t \in [0,T]\big)
\end{equation}
where $\psi$ denotes the fundamental solution of the Hamiltonian system given in Equation \eqref{eq:hamsys-finale}.
\end{defn}
\begin{thm}\label{thm:index}
	Under the previous notation and if assumptions (L1) \& (L2) are fulfilled  the functional $A:H_Z \to \R$ is of regularity class $\mathscr C^2$. 
	
	If $x$ is a critical point of $A_Z$, then  $\iMorse{Z}(x)$ is finite. Moreover there exists a non-negative integer $c(Z) \in \{0,\ldots,n\}$ such that the following equality holds
	\begin{equation}
		\iMorse{Z}(x)=\iMas{L_Z}(z_x)-c(Z).
	\end{equation}
\end{thm}
\begin{proof}
For the proof of this result we refer the reader to  \cite[Theorem 3.4 \& Theorem 2.5]{HS09}.
\end{proof}
\begin{rem}\label{rem:comp_c(Z)}
The integer $c(Z)$ depend upon the boundary conditions. However the authors in \cite[Section 3]{HS09}, computed  $c(Z)$ in some interesting cases.
\begin{itemize}
	\item {\bf (Periodic)\/} $Z\=\Delta\subset \R^n \oplus \R^n$ (where $\Delta$ denotes the graph of the identity in $\R^n$) and $c(Z)=n$
	\item {\bf (Dirichlet)\/} $Z\=Z_1\oplus Z_2=(0)\oplus (0)$ and $c(Z)=n$
\item {\bf (Neumann)\/} $Z\=\R^n\times \R^n$ and $c(Z)=0$
\end{itemize}
We observe that in the case of separate boundary conditions, i.e. $Z= Z_1 \oplus Z_2$, then we get that $c(Z)=\dim (Z_1^\perp \cap Z_2^\perp)$. (Cfr.\cite[Equation (3.28)]{HS09} for further details). 
\end{rem}
\begin{rem}
It is not surprising that in the Dirichlet case and in the Neumann we get the $n$ and $0$. In fact the Morse index of a critical point $x \in H$ of the action $A$ get its  largest possible value with respect to Neumann boundary conditions and the smallest possible value with respect to  Dirichlet  boundary conditions.
\end{rem}
The last result of this section provides a bound on the Maslov index of $z_x$ when $x$ is a minimizer. 
\begin{prop}\label{thm:main-non-oscillation-miimizer} Let $x$ be a minimizer for $A_Z$. Then 
\[
\iMas{L_Z}(z_x)\in \{0,\ldots,n\}.
\]
\end{prop}
\begin{proof}
Being $x$ minimizer, it follow that $\iMorse{Z}(x)=0$ and by Theorem \ref{thm:index}, we get that 
\[
\iMas{L_Z}(z_x)=c(Z).
\]
The conclusion now follows from the fact that $c(Z)\in \{0,\ldots,n\}$. 
\end{proof}
 A direct consequence of Proposition \ref{thm:main-non-oscillation-miimizer} in the case of natural Lagrangian, namely Lagrangian of the form 
\begin{equation}\label{eq:quadratic-Lagrangian}
L(t,q,v)= K(v)- V(q)
\end{equation}
where as usually $K(v)$ and $V(q)$ denote respectively the kinetic and the potential function, is the following result. 
\begin{cor}\label{thm:main-non-oscillation-corollary} Let $L$ be a  $\mathscr C^2$-natural Lagrangian having a   $\mathscr C^2$-concave potential energy and let $x \in H_Z$ be a critical point of $A_Z$. Then 
\[
\iMas{L_Z}(z_x)\in \{0,\ldots,n\}.
\]
\end{cor}
\begin{proof}
Being $L(t,q,v)=K(v)-V(q)$, we get that the Lagrangian function $L$ is $\mathscr C^2$-convex. Let $x\in H$ be a critical point of $A$. By the $\mathscr C^2$-convexity of the Lagrangian, we get that $\iMorse{}(x)=0$ on $H$ and in particular $\iMorse{Z}(x)=0$ for every $Z\subset \R^n\oplus \R^n$. By Theorem \ref{thm:index} \[
\iMas{L_Z}(z_x) = c(Z),
\]
and the conclusion now follows by using Proposition \ref{thm:main-non-oscillation-miimizer}. 
\end{proof}
\begin{rem}
	A common $Z$, often occurring in the applications, is represented by $Z\=Z_1 \oplus (0)$ where $Z_1$ is a linear subspace of $ \R^n$. This subspace directly appears in the classical Sturm non-oscillation theorem \cite[Section 1]{Arn86}.
\end{rem}


\section{Sturm Theory and symplectic geometry}\label{sec:symplectic-sturm-theorems}

The aim of this section is to provide a generalization of the Sturm Alternation and Comparison Theorems proved by Arnol'd in \cite{Arn86} in the case of optical Hamiltonian. The abstract idea behind these results relies on a careful estimates of the H\"ormander (four-fold) index which is used for comparing and estimating the difference of the Maslov indices with respect to two different Lagrangian subspaces.   Our basic reference for this section  is \cite[Section 3]{ZWZ18} and references therein. We stress on the fact that, even in the (classical) case of optical Hamiltonians, we provide new and sharper estimates.
For the sake of the reader, we refer to Section \ref{sec:Maslov} for the main definitions and properties of the intersection indices  as well as for the basic properties of the Lagrangian Grassmannian $\Lagr(V,\omega)$ of the symplectic space $(V,\omega)$.
 
\subsection{A generalization of Sturm Alternation Theorem}
In the $2n$-dimensional symplectic space  $(V,\omega)$, let us consider $\lambda \in \mathscr C^0\big([a,b], \Lagr(V,\omega)\big)$ and $\mu_1, \mu_2 \in \Lagr(V,\omega)$. We now define the two non-negative integers $k_1, k_2$ given by 
\[
\begin{split}
k_1 & \= \min\{\dim \epsilon_1,\dim \epsilon_2\} \textrm{ for }  \epsilon_1\=  \lambda(a) \cap \lambda(b) + \lambda(b) \cap \mu_1  \textrm{ and }  \epsilon_2\= \lambda(a) \cap \lambda(b) + \lambda(b) \cap \mu_2\\
k_2 & \= \min\{\dim \delta_1,\dim \delta_2\} \textrm{ for } \delta_1\= \lambda(a) \cap \mu_1 + \mu_1 \cap \mu_2  \textrm{ and }  \delta_2\=  \lambda(b) \cap \mu_1 + \mu_1 \cap \mu_2
\end{split}
\]
and we let $k\= \max\{k_1,k_2\}$. We are in position to state and to prove the first main result of this section.

\begin{thm}\label{thm:main-1}
Under the previous notation, the following inequality holds: 	
\begin{equation}
\Big \vert\iCLM\big(\mu_2, \lambda(t); t \in [a,b]\big)-\iCLM\big(\mu_1, \lambda(t); t \in [a,b]\big)\Big\vert \leq n- k.
\end{equation}
\end{thm}
\begin{proof}
The proof of this result is  a consequence of Proposition \ref{thm:mainli}, Equation \eqref{eq:triple-coindex-extended} and Remark \ref{rem:molto-utile-stima}. 
First of all, we start to observe that 
\begin{multline}\label{eq:utile}
	\iCLM\big(\mu_2, \lambda(t); t \in [a,b]\big)-\iCLM\big(\mu_1, \lambda(t); t \in [a,b]\big)= s(\lambda(a), \lambda(b); \mu_1, \mu_2)\\=\itriple(\lambda(a), \lambda(b), \mu_2)-\itriple(\lambda(a), \lambda(b), \mu_1) =\itriple(\lambda(a), \mu_1, \mu_2)-\itriple(\lambda(b), \mu_1, \mu_2)
\end{multline}
For $i=1,2$, we denote by  $\pi_{\epsilon_i}$ (resp.  $\pi_{\delta_i}$) the projection onto the symplectic reduction mod $\epsilon_i$ (resp. $\delta_i$). So, we have 
\begin{multline}
\itriple(\lambda(a), \lambda(b), \mu_1)= \noo{+}Q\big(\pi_{\epsilon_1} \lambda(a), \pi_{\epsilon_1} \lambda(b);\pi_{\epsilon_1} \mu_1\big)=\noo{+}Q_{\epsilon_1}\\ \itriple(\lambda(a), \lambda(b), \mu_2)= \noo{+}Q\big(\pi_{\epsilon_2} \lambda(a), \pi_{\epsilon_2} \lambda(b);\pi_{\epsilon_2} \mu_2\big)=\noo{+}Q_{\epsilon_2}\\ 
\itriple(\lambda(a), \mu_1, \mu_2)= \noo{+}Q(\pi_{\delta_1} \lambda(a), \pi_{\delta_1} \mu_1;\pi_{\delta_1} \mu_2)=\noo{+}Q_{\delta_1}\\\
\itriple(\lambda(b), \mu_1, \mu_2)= \noo{+}Q(\pi_{\delta_2} \lambda(b), \pi_{\delta_2} \mu_1;\pi_{\delta_2} \mu_2)=\noo{+}Q_{\delta_2}\
\end{multline}
Since $\dim V_{\epsilon_i}= 2(n-\dim\epsilon_i)$ (resp. $\dim V_{\delta_i}= 2(n-\dim\delta_i)$), it follows that $Q_{\epsilon_i}$ (resp.  $Q_{\delta_i}$) are quadratic forms on $n-\dim\epsilon_i$ (resp. $n-\dim\delta_i$) vector space. So, the inertia indices are integers between $0$ and $n-\dim\epsilon_i$ (resp. $n-\dim\delta_i$). In conclusion, we get that 
\begin{multline}
	0 \le \itriple(\lambda(a), \lambda(b), \mu_1)\le n-\dim\epsilon_1 \le n-k_1, \qquad 0 \le \itriple(\lambda(a), \lambda(b), \mu_2)\le n-\dim\epsilon_2 \le n-k_1\\
	0 \le \itriple(\lambda(a), \mu_1, \mu_2)\le n-\dim\delta_1 \le n-k_2, \qquad 0 \le \itriple(\lambda(b), \mu_1, \mu_2)\le n-\dim\delta_2 \le n-k_2
\end{multline}
By using these inequalities together with Equation~\eqref{eq:utile}, we get that 
\begin{multline}\label{eq:da-mettere-insieme}
\big\vert\iCLM\big(\mu_2, \lambda(t); t \in [a,b]\big)-\iCLM\big(\mu_1, \lambda(t); t \in [a,b]\big)\big\vert=\big\vert\itriple(\lambda(a), \lambda(b), \mu_2)-\itriple(\lambda(a), \lambda(b), \mu_1)\big\vert \\ = \big\vert\noo{+}Q_{\epsilon_2}- \noo{+}Q_{\epsilon_1}\big\vert\le n- k_1 \\
\big\vert\iCLM\big(\mu_2, \lambda(t); t \in [a,b]\big)-\iCLM\big(\mu_1, \lambda(t); t \in [a,b]\big)\big\vert=\big\vert\itriple(\lambda(a), \mu_1, \mu_2)-\itriple(\lambda(b), \mu_1, \mu_2)\big\vert\\=\big\vert \noo{+}Q_{\delta_1}- \noo{+}Q_{\delta_2}\big\vert\le n- k_2
\end{multline}
Putting the inequalities given in Formula~\ref{eq:da-mettere-insieme} all together, we get
\begin{equation}
\Big \vert\iCLM\big(\mu_2, \lambda(t); t \in [a,b]\big)-\iCLM\big(\mu_1, \lambda(t); t \in [a,b]\big)\Big\vert \leq n- k
\end{equation}
where $k=\max\{k_1,k_2\}$. This concludes the proof. 
\end{proof}
\begin{rem}
Loosely speaking, by  Theorem \ref{thm:main-1}, we can conclude that the smaller is the difference of a Lagrangian path with respect to two  Lagrangian subspaces the higher is the intersection between them.   
\end{rem} 
\begin{cor}\label{thm:pigeon}
	Under the notation of Theorem \ref{thm:main-1} and assuming that  $\lambda\cap \mu_1= \lambda(b)\cap \mu_2 = \emptyset$, we get 
	\begin{equation}
\Big\vert\iCLM\big(\mu_2, \lambda(t); t \in [a,b]\big)-\iCLM\big(\mu_1, \lambda(t); t \in [a,b]\big) \Big\vert \leq n- \dim I
\end{equation}
where $I\= \lambda(a)\cap \lambda(b)$.  In particular, if $\lambda$ is a closed path, then we get that
	\[
	 \iCLM\big(\mu_2, \lambda(t); t \in [a,b]\big) = \iCLM\big(\mu_1, \lambda(t); t \in [a,b]\big).
	\] 
\end{cor}
\begin{proof}
We observe that
\[
\lambda(b)\cap \mu_2\subseteq \epsilon_2 \ \textrm{ 	and } \  \lambda(b)\cap \mu_1\subseteq \epsilon_1 \quad \Rightarrow \quad \dim \epsilon_2 \geq \dim I \ \textrm{ 	and } \  \dim \epsilon_1 \geq \dim I.
\]
By this, we get that  $n-k\leq n-k_1$ is less or equal than  $  n - \dim I$. This concludes the proof of the first claim. 

The second claim readily follows by observing that for loops of Lagrangian subspaces, we have $\dim I=n$. 
\end{proof}
\begin{rem}
It is worth noticing that in the case of Lagrangian loops, the $\iCLM$-index is actually independent on the vertex of the train. This property was already pointed out by   Arnol'd in his celebrated paper \cite{Arn67}.
\end{rem}
\begin{cor}\label{thm:pigeon-2}
	Under notation of Theorem \ref{thm:main-1} and if  $\mu_1\cap \lambda(a)=\mu_1 \cap \lambda(b)=\emptyset$, then we have
	\begin{equation}
\Big\vert
\iCLM\big(\mu_2, \lambda(t); t \in [a,b]\big) -
\iCLM\big(\mu_1, \lambda(t); t \in [a,b]\big) 
\Big\vert \leq n- \dim J
\end{equation}
where $J\=  \mu_1 \cap \mu_2$. 
\end{cor}
\begin{proof}
We observe that
\[
\mu_1\cap \mu_2\subseteq \delta_1 \ \textrm{ 	and } \  \mu_1\cap \mu_2\subseteq \delta_2 \quad \Rightarrow \quad \dim \delta_1 \geq \dim J \ \textrm{ 	and } \  \dim \delta_2 \geq \dim J.
\]
By this, we get that $n-k \leq n-k_2$ is  less or equal than  $  n - \dim J$. 
\end{proof}
\begin{rem}
We observe that if the four Lagrangians $\lambda(a), \lambda(b), \mu_1, \mu_2$ are mutually transversal, then $k=0$. Thus in this case the modulus of the difference of the Maslov indices computed with respect to two (distinguished) Lagrangian is bounded by $n$.
\end{rem}
\begin{rem}
We observe that Corollary \ref{thm:pigeon} and  Corollary \ref{thm:pigeon-2}  are well-known. More precisely  Corollary \ref{thm:pigeon} agrees with  \cite[Corollary 3.4]{JP09}  and Corollary \ref{thm:pigeon-2} corresponds to  \cite[Proposition 3.3]{JP09}. As by-product of the previous arguments we get that the inequalities proved by authors in aforementioned paper were not sharp. 

It is worth noticing that the proof provided by authors is completely different from the one given in the present paper and it mainly relies on a careful estimate of  the inertial indices of symmetric bilinear forms obtained by using the atlas of the Lagrangian Grassamannian and its transition functions. 
\end{rem}

\begin{thm}\label{thm:main-3}
Let 	$ L_0, L_1, L_2 \in \Lagr(V,\omega)$, $ \psi \in \PT{V, \omega}$ and for every $ t \in [0,T]$, we let 
\[
\ell_1(t)\= \psi(t)L_1, \quad \ell_2(t)\= \psi(t)L_2 \textrm{ and } \mu_0(t)\= \psi^{-1}(t)L_0.
\]
Thus, we have 
\begin{equation}
\Big \vert
\iCLM\big(L_0,\ell_1(t); t \in [a,b]\big) - \iCLM\big(L_0,\ell_2(t); t \in [a,b]\big)
\Big\vert 
\leq n- k
\end{equation}
where $k\=\min\{\dim \epsilon_a,\dim \epsilon_b\}$ and where 
$\epsilon_a\= L_1 \cap L_2 + L_2 \cap L_0 $ while $\epsilon_b\=  L_1 \cap L_2 + L_2 \cap \mu_0(b).$
\end{thm}
\begin{proof}
By taking into account the symplectic invariance of the $\iCLM$-index,  we get 
\[
\iCLM\big(L_0, \ell_1(t); t \in [a,b]\big) =
\iCLM\big(\psi(t)^{-1}L_0,  L_1; t \in [a,b]\big)= \iCLM\big(\mu_0(t),  L_1; t \in [a,b]\big)
\]
and
\[
\iCLM\big(L_0, \ell_2(t); t \in [a,b]\big)= 
\iCLM\big(\psi(t)^{-1} L_0,L_2; t \in [a,b]\big)= \iCLM\big(\mu_0(t),L_2; t \in [a,b]\big).
\]
Moreover 
\begin{multline}
\iCLM\big(\mu_0(t),  L_2; t \in [a,b]\big)-
\iCLM\big(\mu_0(t),  L_1; t \in [a,b]\big)=
s\big(L_1,L_2; L_0, \mu_0(b)\big)\\
=\itriple(L_1,L_2,L_0) - \itriple\big(L_1,L_2, \mu_0(b)\big).
\end{multline}
The proof now immediately follows by theorem \ref{thm:main-1}. 
\end{proof}
By restricting  Theorem \ref{thm:main-1} to curves of Lagrangian subspaces induced by the evolution of a fixed Lagrangian under the phase flow of a linear Hamiltonian system we get a generalization of the Sturm Alternation Theorem proved by Arnol'd in \cite{Arn86}. More precisely, let us consider the linear Hamiltonian system 
\begin{equation}\label{eq:222}
	z'(t)= J_0 B(t) z(t), \qquad t \in[0,T].
\end{equation}
 Then the following result holds. 
\begin{thm}{\bf [Sturm Alternation Theorem]\/}\label{thm:main-corollary-alternation}. 
Let $L,L_1,L_2\in \Lagr(n)$ and  we  set  $ \ell(t)\=\phi(t)L$ where $\phi$ denotes the fundamental solution of Equation \eqref{eq:222}. Then we get 
\begin{equation}
\Big \vert
\iCLM\big(L_2, \ell(t); t \in [0,T]\big)-
\iCLM\big(L_1, \ell(t); t \in [0,T]\big)
\Big\vert \leq n- k
\end{equation}
where $k\= \max\{k_1,k_2\}$ and 
\[
\begin{split}
k_1 & \= \min\{\dim \epsilon_1,\dim \epsilon_2\} 
\textrm{ for } 
\epsilon_1 \= L \cap \ell(T) + \ell(T) \cap \L_1
\textrm{ and }
\epsilon_2 \= L \cap \ell(T) + \ell(T) \cap L_2  \\
k_2 & \= \min\{\dim \delta_1,\dim \delta_2\} 
\textrm{ for } 
\delta_1  \=  L \cap L_1 + L_1 \cap L_2 
\textrm{ and }
\delta_2\= \ell(T) \cap L_1 + L_1 \cap L_2   .
\end{split}
\]
\end{thm}
\begin{rem}
We stress on the fact that in the aforementioned paper, Arnol'd proved the Alternation Theorem for the class of quadratic Hamiltonian functions  that are optical with respect to the two distinguished Lagrangian subspaces $L_1$ and $L_2$. In the  classical formulation, author provides a bound on the difference of non-transversality moments of the evolution of a Lagrangian path with respect to two distinguished Lagrangian subspaces. 
\end{rem}


\subsection{Iteration inequalities for periodic boundary conditions}

In this section we provide some simple estimates on the Conley-Zehnder index ofwhich can be obtained directly from Theorem \ref{thm:main-1}. 

Given a  symplectic space $(V,\omega)$, we consider the direct sum $V^2\= V \oplus V$, endowed with the symplectic form  $\omega^2\= -\omega \oplus \omega$, defined as follows 
\[
 \omega^2((v_1, v_2),(w_1,w_2))= -\omega(v_1, v_2) +\omega(w_1, w_2), \qquad \textrm{ for all } v_1, v_2, w_1, w_2 \in V
\]
and we recall that  
\[
 \quad \psi \in \mathscr C^0\big([a,b], \Sp(V,\omega)\big)\quad \Rightarrow \quad  \Graph\psi\in \mathscr C^0\big([a,b],  \Lagr(V^2, \omega^2)\big),
\]
and $\Delta$ is the diagonal subspace of $V\oplus V$.
\begin{defn}\label{def:CZ}
 Let  $\psi\in \mathscr C^0\big([a,b], \Sp(V,\omega)\big)$.  The {\em generalized Conley-Zehnder index\/} of $\psi$ is the integer $\iCZ(\psi)$ defined as follows
 \[
  \iCZ(\psi(t); t \in [a,b])
  \=
  \iCLM\big(\Delta, \Graph \psi(t); t \in [a,b] \big).
 \]
\end{defn}
\begin{rem}
We observe that the Conley-Zehnder index was originally defined for symplectic paths having non-degenerate final endpoint meaning that $\Graph \psi(b) \cap \Delta=\{0\}$. We emphasize that, for curves having degenerate endpoints with respect to $\Delta$ there are several conventions for how the endpoints contribute to the Maslov index. For other different choices we refer the interested reader to \cite{RS93, LZ00, DDP08} and references therein. 
\end{rem}
\begin{lem}\label{thm:maslov-pairs}
Let $L_1,L_2 \in \Lagr(V,\omega)$ and $\psi \in \mathscr C^0\big([a,b], \Sp(V, \omega)\big)$. Then 
\[
 \iCLM\big(L_1\oplus L_2, \Graph\psi(t); t \in [a,b]\big) = \iCLM\big(L_2,  \psi(t ) L_1; t \in [a,b]\big).
 \]
\end{lem}
\begin{proof}
The proof of this result follows by  \cite[Theorem 3.2]{RS93} and Equation \eqref{eq:clm-rs}
\end{proof}

\begin{thm}\label{thm:main-2}
Let $(V,\omega)$ be a $2n$-dimensional symplectic space,  $L, L_0 \in \Lagr(V, \omega)$,  $\psi\in \PT{V,\omega}$ and let    $\ell \in \mathscr C^0\big([0,T], \Lagr(V,\omega)\big) $ be pointwise defined by  $\ell(t)\=\psi(t)\, L$. 
Then the following inequality holds
\begin{equation}
\Big \vert\iCLM\big(L_0, \ell(t); t \in [a,b]\big) -
\iCZ\big(\psi(t); t \in [a,b]\big)\Big\vert \leq 2n- \dim \epsilon
\end{equation}
where $\epsilon$ is the even dimensional subspace defined by $\epsilon\=\Graph P \cap \Delta + \Delta \cap (L\oplus L_0)$  with $P=\psi(T)$.
\end{thm}
Before proving this result, we observe that the maximal dimension of the isotropic subspace $\epsilon$ is an even number less or equal than $2n$. This is for instance the case in which $P=\Id$.
\begin{proof}
By invoking Lemma \ref{thm:maslov-pairs}, we start to observe that  
\begin{equation}
\iCLM (L_0, \ell(t);t \in [0,T])=
\iCLM (L_0,\psi(t)\,L;t \in [0,T])\\
=
\iCLM\big(L\oplus L_0,\Graph\psi(t);t \in [0,T]\big)
\end{equation}
and by Definition \ref{def:CZ}, we know that 
$\iCZ(\psi(t); t \in [0,T])= \iCLM\big(\Delta, \Graph \psi(t); t \in [0,T]\big)$. 
Summing up, we get  
\begin{multline}
\iCLM\big(L_0, \ell(t); t \in [0,T]\big)- \iCZ\big(\psi(t); t \in [0,T]\big)\\=
\iCLM\big(L\oplus L_0, \Graph\psi(t); t \in[0,T]\big)- \iCLM\big(\Delta,\Graph\psi(t); t \in [0,T]\big)\\
=
s \big( \Delta, \Graph P;\Delta, L\oplus L_0\big)
=-\itriple\big(\Graph P, \Delta, L\oplus L_0\big)
\end{multline}
where in the last equality we used Lemma \ref{thm:properties}, (I).  
We observe that $\itriple\big(\Graph P, \Delta, L\oplus L_0\big)$ is equal to the extended coindex of a quadratic form on a Lagrangian subspace of the reduced space $V_\epsilon\=\epsilon^\omega/\epsilon$ (see Equations \eqref{eq:triple-coindex-extended}). Thus the sum of all inertial indices is bounded from above by  $1/2 \dim V_\epsilon$ which is equal to $2n-\dim \epsilon$.
\end{proof}
\begin{rem}
For an explicit computation of the term $\itriple\big(L\oplus L , \Delta, \Graph(P)\big)$, we refer the interested reader to \cite{Por08, FK14} and references therein. 
\end{rem}
\begin{defn}\label{def:Lagrangian-Maslov-index}
Given $L \in \Lagr(V,\omega)$, we term the {\em $L$-Maslov index\/} the integer given by  
\begin{equation}
\iL\big(\psi(t), t \in [a,b] \big)\= \iCLM\big(L \oplus L, \Graph\psi(t); t \in [a,b]\big).
\end{equation}
\end{defn}
As direct consequence of Theorem \ref{thm:main-2} and Definition \ref{def:Lagrangian-Maslov-index} we get the following.
\begin{lem}
Under notation of Theorem \ref{thm:main-2}, the following inequality holds:
\begin{equation}\label{eq:stime-cz-l}
\Big|\iL\big(\psi(t), t \in [a,b]\big) -\iCZ\big(\psi(t); t \in [a,b]\big) \Big| \leq 2n-\dim W \leq n
\end{equation}
where $W\=\Graph P \cap \Delta+ (L\oplus L) \cap \Delta$. 
\end{lem}
\begin{proof}
The proof of the first inequality in Equation \eqref{eq:stime-cz-l} comes directly by Theorem \ref{thm:main-2}.   The second inequality follows by observing that $W \supseteq (L\oplus L) \cap \Delta $ and thus  $\dim W \geq n$.
 \end{proof}
Typically in concrete applications, one is faced with the problem of estimating the difference of the $\iCLM$-indices of two different Lagrangian curves with respect to a distinguished Lagrangian subspace. These Lagrangian curves are nothing but the evolution under the phase flow of two distinguished Lagrangians.


We set 
\begin{equation}
D_\omega(M)\= (-1)^{n-1}\overline \omega^n \det(M- \omega \Id), \qquad \omega \in \mathbb U, \ M \in \Sp(2n,\R).
\end{equation}
Then for any $\omega \in \mathbb U$, let us consider the hypersurface in $\Sp(2n)$ defined as  
\begin{equation}
\Sp_\omega^0(2n,\R)\= \Set{M \in \Sp(2n,\C)| D_\omega(M)=0}. 
\end{equation}
As proved by Long (cf. \cite{Lon02} and references therein), for any $M \in \Sp(2n)_\omega^0$, we define a co-orientation of $\Sp(2n)_\omega^0$ at $M$ by the positive direction $\frac{d}{dt}|_{t=0} M e^{tJ}$ of the path $ M e^{tJ}$ with $t \geq 0$ sufficiently small. Let 
\begin{equation}
\Sp_\omega^*(2n,\R)\= \Sp(2n,\R)\setminus \Sp_\omega^0(2n,\R).
\end{equation}
Given $\xi, \eta \in \mathscr C^0\big([0,T], \Sp(2n,\R)\big)$ with $\xi(T)=\eta(0)$, we define the concatenation of the two paths as 
\begin{equation}
(\eta*\xi)(t)= \begin{cases}
 \xi(2t) & 0 \leq t \leq T/2\\
 \eta(2t-T) & T/2 \leq t \leq T
 \end{cases}.
\end{equation}
For any $n \in \N$, we define the following special path $\xi_n  \in \PT{2n}$ as follows 
\begin{equation}
\xi_n(t)= \begin{bmatrix}
 2-\dfrac{t}{T} & 0\\
 0 & \left(2-\dfrac{t}{T}\right)^{-1}
 \end{bmatrix}^{\diamond n}\qquad 0 \leq t \leq T
\end{equation}
where $\diamond$ denotes the diamond  product of matrices. (Cf. \cite{Lon02} for the definition). 
\begin{defn}
For any $\omega \in \mathbb U$ and $\psi \in \PT{2n}$, we define 
\[
\nu_\omega(\psi)\= \dim \ker_\C\big(\psi(T)-\omega \Id\big),
\]
and the {\em $\omega$-Maslov type index\/}  $\iota_\omega(\psi)$ given by setting
\[
\iota_\omega(\psi)\=\Big[e^{-\varepsilon J}\psi*\xi_n: \Sp(2n)_\omega^0\Big]
\] 
that is the intersection index between the path $e^{-\varepsilon J}\psi*\xi_n$ and the transversally oriented hypersurface $\Sp_\omega^0(2n)$.  
\end{defn}
We now set, for any $\psi \in \PT{2n}$,
\[
 \psi_{\kappa+1}(t) = \psi(t-   \kappa T) P^\kappa, \qquad \kappa  T   \leq t \leq (\kappa +1)  T
\]
where $P\= \psi(T)$ and we define the {\em $m$-th iteration\/} $\psi^m \in \mathscr C^0\big([0,mT], \Sp(2n,\R)\big)$ of $\psi$ as follows 
\[
\psi^m(t) \= \psi_{\kappa+1}(t) \textrm{ for } \kappa  T   \leq t \leq (\kappa+1) T \textrm{ and } \kappa=0,1, \dots, m-1.
\]
Based on the index function $\iota_\omega$, Long established (cfr. \cite{Lon02} and references therein) a Bott-type iteration formula for any path $\psi \in \PT{2n}$ that reads as follows 
\begin{multline}\label{eq:Bott-type}
\iota_z\big(\psi^m(t), t \in [0,mT]\big)= \sum_{\omega^m=z} \iota_\omega\big(\psi(t), t \in [0,T]\big) \textrm{ and } \\\nu_z\big(\psi^m(t), t \in [0,mT]\big)= \sum_{\omega^m=z}\nu_\omega\big(\psi(t), t \in [0,T]\big).
\end{multline}
\begin{lem}\label{thm:cor2.1LZ00}
For any $\psi \in \PT{2n}$, we have 
\begin{multline}
\iota_1\big(\psi(t); t \in [0,T]\big)+n   = 
\iCLM\big( \Delta, \Graph\psi(t); t \in [0,T]\big)\\
\iota_\omega\big(\psi(t); t \in [0,T]\big)  =\iCLM\big(\Delta_\omega,\Graph\psi(t); t \in [0,T]\big), \qquad \omega \in \mathbb U \setminus\{1\}
\end{multline}
where $\Delta_\omega \= \Graph(\omega \Id)$.
\end{lem}
\begin{proof}
For the proof of this result, we refer the interested reader to \cite[Corollary 2.1]{LZ00}. 
\end{proof}
Given   $L \in \Lagr(V,\omega)$ and $\psi \in \PT{2n}$, we define the continuous curve $\ell^m: [0,m T  ]\to \Lagr(n)$  as 
\[
\ell^m(t) \= \psi^m(t)L.
\]

By the affine scale invariance of the Maslov index, for any given  $L \in \Lagr(n)$, we get
\[
\iCLM\big(L, \psi_{\kappa+1}(t)L;t \in [\kappa  T  , (\kappa+1)  T  ]\big)=  
\iCLM\big(L,\psi(t) P^\kappa L;t \in [0,T ]\big),\quad \kappa \in\{0, \dots,m-1\}.
\]
By taking into account the additivity property of the Maslov index under concatenations of paths and Lemma \ref{thm:maslov-pairs}, we infer
\begin{multline}\label{eq:riduzioneiterata}
\iCLM(L,\ell^m(t); t \in [0, mT])\\
 = \sum_{\kappa=0}^{m-1}
\iCLM(L,\psi_{\kappa+1}(t)L;t \in [\kappa T,(\kappa+1)T])= 
   \sum_{\kappa=0}^{m-1}
\iCLM(L,\psi (t) P^\kappa L;t \in [0,T ])\\= \sum_{\kappa=0}^{m-1} 
\iCLM \big( P^\kappa L\oplus L, \Graph \psi(t); t \in [0,T]\big).
 \end{multline}
 In particular, if $L$ is $P$-invariant (namely $PL \subseteq L$), then we have 
 \begin{multline}\label{eq:riduzioneiterata-2}
 \iL(\psi(t), t \in [0,mT])\\=
 \iCLM(L,\ell^m(t); t \in [0, mT]) =  
 \sum_{\kappa=0}^{m-1} \iCLM\big(L\oplus L,\Graph \psi(t); t \in [0,T]\big) \\ 
  = m\, \iCLM \big(L\oplus L,\Graph \psi(t); t \in [0,T]\big) =  m\,\iL\big( \psi(t); t \in [0,T]\big).
 \end{multline}
\begin{prop}\label{thm:stima-conley-iterate}
Let $\psi \in \PT{2n}$ and $m \in \N$. Then 
\[
	k_1 -k_m\geq   \iCZ(\psi^m(t); t \in [0, mT])-m\,\iCZ(\psi(t); t \in [0, T])\geq - (m-1)\,\cdot (2n-k_1)
	\]
where $k_i =\dim (\Graph P^i\cap \Delta)$.
\end{prop}
\begin{proof}
By invoking the Bott type  iteration formula given in Equation  \eqref{eq:Bott-type}, Definition \ref{def:CZ} and Lemma \ref{thm:cor2.1LZ00}, we get
\[
 \iCZ(\psi^m(t); t \in [0, mT])= \iCZ(\psi(t);t\in [0, T])+ \sum_{\substack{\omega^m=1\\ \omega \neq 1}}\iCLM\big( \Delta_\omega,\Graph\psi(t); t \in [0, T]\big).
\]
For every $\omega \in \mathbb U$, using Lemma \ref{thm:properties}, we have 
\begin{multline}
	\iCLM\big(\Delta_\omega,\Graph \psi(t);t \in [0, T]\big)-\iCLM\big(\Delta, \Graph \psi(t); t \in [0, T]\big)= s(\Delta, \Graph P;\Delta,\Delta_\omega)\\
	=-\itriple(\Graph P,\Delta, \Delta_\omega).
	\end{multline}
Summing up, we finally get 
\begin{equation}\label{eq:buona}
	\iCZ(\psi^m(t); t \in [0, mT])= m \, \iCZ(\psi(t);t \in [0, T]) - \sum_{\substack{\omega^m=1\\ 
			\omega \neq 1}}\itriple\big(\Graph P, \Delta, \Delta_\omega\big). 
	\end{equation}
Now, for every root of unit $\omega_i$, by using analogous arguments as given in the proof of Theorem~\ref{thm:main-1}, we get that the triple index  $\itriple\big(\Graph P, \Delta, \Delta_{\omega_i}\big)$  is equal to the extended coindex of a quadratic form on a $(2n-\dim \epsilon_i)$-dimensional vector space where 
$\epsilon_i\= \Delta\cap \Delta_{\omega_i} + \Delta \cap \Graph P=\Delta \cap \Graph P$. Set $k_1=(\dim \Delta\cap \Graph P)$ , then we get that  
\begin{equation}\label{eq:dis-ultima}
\iCZ(\psi^m(t); t \in [0, mT])-m\,\iCZ(\psi(t); t \in [0, T])\geq -(m-1)( 2n- k_1). 
\end{equation}
Furthermore, use \eqref{eq:triple}, we have $\itriple\big(\Graph P, \Delta, \Delta_\omega\big)\ge \dim (\Delta_{\omega }\cap \Graph (P))$.
It follows that 
\[
\sum_{\substack{\omega^m=1\\ 
		\omega \neq 1}}\itriple\big(\Graph P, \Delta, \Delta_\omega\big)\geq \dim\ker (P^m-\Id)-\dim\ker (P-\Id)
\]
This concludes the proof.
\end{proof}
\begin{rem}
For an analogous  estimate, we refer  the interested reader to \cite[Corollary 3.7, Equation  (12)]{DDP08}. We remark that the  estimate provided in Proposition~\ref{thm:stima-conley-iterate} coincides with the one proved by authors in   \cite[Equation (19), Theorem 3, pag.213]{Lon02} with completely different methods once observed that $\iCZ(\psi(y), t \in [0,T]=i_1(\psi)+n$ where $i_1$ is the index appearing in the aforementioned book of Long. 
\end{rem}
%


\section{Optical Hamiltonian and Lagrangian plus curves}\label{subsec:optical}

This section is devoted to discuss a {\em monotonicity property of the crossing forms\/} for a path of Lagrangian subspaces with respect to a distinguished  Lagrangian subspace $L_0$; such a property  is usually termed {\em $L_0$-positive\/} (respectively $L_0$-negative) or {\em $L_0$-plus\/} (respectively {\em $L_0$-minus\/}) property. We start with the following definition. 
\begin{defn}\label{def:plus-curve}
Let $L_0\in \Lagr(V,\omega)$. A curve $\ell:[a,b] \to \Lagr(V,\omega)$ is termed a {\em $L_0$-plus curve\/} or {\em $L_0$-positive curve\/}  if, at each crossing instant $t_0\in [a,b]$, the crossing form $\Gamma(\ell(t), L_0, t_0)$ is  positive definite. 

If $\ell$ is a $L_0$-plus and if  $t_0 \in [a,b]$ is a crossing instant,  we define the {\em multiplicity\/} of the crossing instant $t_0$, the positive integer
 \[
 \mul(t_0)\=\dim\big(\ell(t_0)\cap L_0\big).
\]
 \end{defn}
 \begin{rem}
 We observe that an analogous definition holds for {\em $L_0$-minus curves\/} just by replacing plus by  minus. 	
 \end{rem}

 \begin{rem}
We stress on the fact that the {\em plus condition\/} strongly depends on the train $\Sigma(L_0)$. In fact, as we shall see later,   a curve of Lagrangian subspaces could be a plus curve with respect to a train but not with respect to another (or even worse with respect to any other). 
\end{rem}
Thus for $L_0$-plus curves we get the following result.
\begin{lem}\label{thm:maslov-positive}
Let  $\ell \in \mathscr C^1\big([a,b], \Lagr(V,\omega))$ be a {\em $L_0$-plus curve\/}.  Then we have: 
\begin{equation}\label{eq:Maslov-plus}
  \iCLM\big(L_0,\ell(t); t \in [a,b]\big)=\mul(a)+\sum_{\substack{t_0 \in \ell^{-1}\Sigma(L_0)\\ t_0 \in ]a,b[}}\mul(t_0).
\end{equation}
\end{lem}
\begin{proof}
We observe that if $\ell$ is a $L_0$-plus curve then 
\[
\sgn\Gamma(\ell, L_0, t_0)=\coindex \Gamma(\ell, L_0, t_0)= \dim\big(\ell(t_0)\cap L_0\big).
\]
Since $\ell$ is a plus curve, each crossing instant is non-degenerate and in particular isolated. So, on a compact interval are in finite number. 
We conclude the proof using Equation \eqref{eq:iclm-crossings}.  
\end{proof}
In this paragraph we provide sufficient conditions on the Hamiltonian function in order  the lifted Hamiltonian flow at the Lagrangian Grassmannian level is a plus curve with respect to a distinguished Lagrangian subspace.  

On the  symplectic space  $(\R^{2n}, \omega_0)$, let  $H:[0,T] \times \R^{2n} \to\R$  be  a (smooth) Hamiltonian and let us consider the first order  Hamiltonian system  given by 
\begin{equation}\label{eq:Ham-sys} 
z'(t)= J_0 \nabla H\big(t, z(t)\big), \qquad t \in [0,T],
\end{equation}
($\omega_0$ and $J_0$ have been introduced at page \pageref{eq:standard-form}).
By linearizing  Equation   \eqref{eq:Ham-sys} along a solution $z_0$, we get the system
\begin{equation}\label{eq:-linHam-sys} 
w'(t)= J_0 B(t) w(t),	\qquad t \in [0,T]
\end{equation}
where 
\begin{equation}\label{eq:hessian-ham}
B(t) := D^2H\big(t, z_0(t)\big) =\begin{bmatrix}
 H_{pp}(t)	& H_{pq}(t)\\
 H_{qp}(t)& H_{qq}(t)
 \end{bmatrix}
\end{equation}
We denote by $\psi$  the fundamental solution of the Hamiltonian system given at Equation \eqref{eq:-linHam-sys}. 
\begin{rem}
 We observe that if $H$ is quadratic and $t$-independent, the  linear Hamiltonian vector field in Equation \eqref{eq:-linHam-sys} is $t$-independent, i.e. $B(t)=B$. In this particular case,  we get $\psi(t)= \exp(tJ_0 S)$. 
\end{rem}
\begin{defn}\label{def:optical-lagrangian}
Let  $L_0, L \in \Lagr(n)$ and let $\ell : [0,T]\to  \Lagr(n)$ be defined by $\ell(t)\= \psi(t)\,L$. The Hamiltonian $H$ is termed {\em $L_0$-optical\/} or {\em $L_0$-positively twisted\/} if the curve  $t \mapsto\ell(t)$ is a $L_0$-plus curve.
\end{defn} 
Some important special classes of $L_0$-optical Hamiltonians where $L_0$ is the Dirichlet (resp. Neumann) Lagrangian is represented by Hamiltonian having some convexity  properties with respect to the momentum (resp. configuration) variables. 

\begin{prop}\label{thm:buona}
Let $H:\R^{2n}\to\R$ be a $\mathscr C^2$-convex Hamiltonian and let $z_0$ be a solution of the Hamiltonian system given in Equation  \eqref{eq:Ham-sys}. Then we get that $H$ with respect to the 
\begin{enumerate}
\item momentum variables is $L_D$-optical
\item configuration variables is $L_N$-optical.
\end{enumerate}
\end{prop}
\begin{proof}
We prove only the first statement, being the second completely analogous. Given $L \in \Lagr(n)$, let us consider the Lagrangian curve pointwise defined by  $\ell(t)\= \psi(t)L$. Let $t_0$ be a crossing instant for $\ell$ with respect to the Dirichlet Lagrangian $L_D$. By using Equation \eqref{eq:plus-curve-control} and Equation \eqref{eq:hessian-ham}, we get that
\begin{equation}\label{eq:Gamma-L_D}
	\Gamma\big(\ell(t), L_D, t_0 \big)[w]=	\langle B(t_0) w,w \rangle= \langle H_{pp}(t_0) y, y\rangle, \qquad \forall\, w=\begin{bmatrix}y\\0\end{bmatrix} \in\ell(t_0)\cap L_D.
	\end{equation}
Since $H$ is $\mathscr C^2$ convex in the $p$-variables, it follows that the crossing form $\Gamma$ given in Equation \eqref{eq:Gamma-L_D} is positive definite. The conclusion now follows by the arbitrarily  of $t_0$. 
\end{proof}
\begin{cor}\label{thm:plus-curve-schur}
Let $H:\R^{2n}\to\R$ be a $\mathscr C^2$-strictly convex Hamiltonian function and let $z_0$ be a solution of the Hamiltonian system given in Equation  \eqref{eq:Ham-sys}.  Then $H$ is $L_0$-optical with respect to every $ L_0\in \Lagr(n)$. 
\end{cor}
\begin{proof}
In fact, since $H$ is $\mathscr C^2$-strictly convex, this in particular implies that $B(t)=D^2H\big(t, z_0(t)\big)$ is positive definive and hence every restriction is positive definite. The conclusion now follows directly by using once again  Equation \eqref{eq:plus-curve-control}.
\end{proof}
\begin{rem}
We consider the  Hessian of $H$ along a solution $z_0$ of the Hamiltonian system given in Equation  \eqref{eq:Ham-sys}, given by Equation \eqref{eq:hessian-ham}
and we observe that in terms of the block matrices entering in the Hessian of $H$, the condition for $H$ to be $\mathscr C^2$-strictly convex is equivalent to
 \begin{enumerate}
\item $H_{pp}(t)$ is positive definite (in particular invertible);
\item $ H_{qq}(t)- H_{qp}(t)H_{pp}(t)^{-1}H_{pq}(t) $
is positive definite. 
\end{enumerate}
The equivalence readily follows by the characterization of positive definiteness of a block matrices in terms of the Schur's complement.   Thus, in general, if the Lagrangian $L$ given in Definition~\eqref{def:optical-lagrangian} is  not in a special position with respect to $L_D$ and $L_N$, the opticality property strongly depends upon the all blocks appearing in the Hessian of $H$. 
\end{rem}
We are now in position to prove the Sturm non-oscillation theorem. 
\begin{thm}{\bf [Sturm Non-Oscillation]\/}\label{thm:sturm-nonoscillation}
Let $H: [0,T] \oplus \R^{2n} \to \R$ be a $\mathscr C^2$ Legendre convex natural quadratic Hamiltonian of the form
\begin{equation}\label{eq:quadratic-Hamiltonian}
H(p,q)= \dfrac12\Big[\langle B(t)p,p \rangle+\langle A(t)q,q\rangle\Big],
\end{equation}
where $A,B : [0,T] \to \Sym(n)$ (with $B(t)$ positive definite for every $y \in [0,T]$).
Let $\psi$ be the fundamental solution of the linearized system given in Equation~\eqref{eq:-linHam-sys},  
$L_0 \in \Lambda(n)$, and  $\ell_0(t)\=\psi(t)L_0$. Setting  $\mul(t_0)\= \dim \big(\ell(t_0)\cap L_D\big)$, then we get that 
 \[
 \sum_{t_0 \in [0,T]}\mul(t_0)\leq n
 \]
\end{thm}
\begin{proof}
 Let $x$ be the critical point (with Dirichlet boundary conditions) of the action functional corresponding to the solution $z_0$. Then the Morse index of $x$ is 0, since the (natural) Lagrangian $L$ corresponding to the Hamiltonian $H$ is  $\mathscr C^2$  convex.
In particular by Theorem \ref{thm:index}, we have
\[
\iMas{L_Z}(z_0) = c(Z).
\]
Here $Z= (0) \oplus (0)$, $L_Z=L_D$, and  by taking into account Remark~\ref{rem:comp_c(Z)} we get that  $c(Z)=n$. Then 
\(
\iMas{L_Z}(z_0) = n
\)
and by Definition \ref{def:I_Morse} we have
\begin{equation}\label{eq:ultima}
	\iCLM(L_D \oplus L_D,\Graph(\psi(t));t\in [0,T]) = 
\iCLM(L_D,\psi(t)L_D;t\in [0,T]) = n
\end{equation}
Note that  $\L_D\cap (\psi(0)L_D )=n$ and the Hamiltonian  is $L_D$-optical .
	By lemma \ref{thm:maslov-positive}, we have 
	\begin{equation}\label{eq:nondegenerate}
	L_D\cap (\psi(T)L_D)=\{0\}.
	\end{equation}
From Definition \ref{def:hormander} and Proposition \ref{thm:mainli} we get
\begin{multline}
\iCLM(L_D \oplus L_D,\Graph(\psi(t));t\in [0,T]) - \iCLM(L_0 \oplus L_D, \Graph(\psi(t));t\in[0,T]) \\
= s(\Graph(\Id),\Graph(\psi(T)); L_0\oplus L_D,L_D\oplus L_D)\\
= \iota(\Graph(\Id),L_0 \oplus L_D,L_D \oplus L_D) - \iota(\Graph(\psi(T)),L_0 \oplus L_D,L_D \oplus L_D) .
\end{multline}
By  \cite[ Equation (1.17)]{HWY18}, we have
\[
\iota(\Graph(\Id),L_0\oplus L_D,L_D\oplus L_D)=n-\dim(L_0\cap L_D)+\iota(L_0,L_D,L_D)= n-\dim(L_0\cap L_D)
\]
where the last equality follows by \cite[Corollary 3.14]{ZWZ18}.
By equation \eqref{eq:triple} and \eqref{eq:nondegenerate}, we have
	\begin{multline}\iota(\Graph(\psi(T)),L_0\oplus L_D,L_D\oplus L_D)\le 2n-\dim\left(\psi(T)\cap(L_0\oplus L_D)\right) \\
	-\dim((L_0\oplus L_D) \cap (L_D\oplus L_D))+\dim\left(\Graph(\psi(T))\cap(L_0\oplus L_D)\cap (L_D\oplus L_D)\right)\\
	=2n-\dim((\psi(T)L_0)\cap L_D)- (n+\dim(L_0\cap L_D))+((\psi(t)( L_D\cap L_0))\cap L_D)\\
	=n-\mul(T)-\dim(L_0\cap L_D).
	\end{multline}
   We get
 \begin{multline}
\iCLM(L_D \oplus L_D,\Graph(\psi(t));t\in [0,T]) - \iCLM(L_0 \oplus L_D, \Graph(\psi(t));t\in[0,T]) \\\geq n-\dim(L_0 \cap L_D) - \left(n-\dim(L_0 \cap L_D-\mul(T))\right) =\mul(T).
\end{multline}
By this inequality and by Equation~\eqref{eq:ultima}, we get that 
\begin{multline}
\iCLM(L_0\oplus L_D,\Graph(\psi(t));t\in [0,T])\\
=\iCLM(L_D,\psi(t)L_0;t\in[0,T])\le \iCLM(L_D,\psi(t)L_D;t\in [0,T])-\mul(T) = n-\mul(T).
\end{multline}
The thesis follows by observing that in the case of positive curves, it holds that
\[
\iCLM(L_0\oplus L_D,\Graph(\psi(t));t\in [0,T])= \sum_{t_0 \in [0,T)} \mul(t_0).
\]
\end{proof}
\begin{rem}
It is worth noticing that, in fact 
\begin{multline}
\mul(0)\=\dim (L_0\cap L_D)\le \iCLM(L_0\oplus L_D,\Graph(\psi(t));t\in [0,T])\\=\iCLM(L_D,\psi(t)L_0;t\in[0,T])\le n.
\end{multline}
\end{rem}
Now, since the natural Hamiltonian is $\mathscr C^2$ Legendre convex, as direct consequence of Proposition~\ref{thm:buona}, we get that the curve $t\mapsto \ell_0(t)$ is $L_D$-plus and by using Lemma \ref{thm:maslov-positive}, the local contribution to the $\iCLM$-index is through the multiplicity. This concludes the proof.
\begin{rem}
By using the suggestive original Arnol'd language, the Sturm non-oscillation theorem given in Theorem~\ref{thm:sturm-nonoscillation} could be rephrased by stating that  
\begin{quote}
\underline{Nonoscillation Theorem.} If the potential energy is nonpositive, then the number of moments of verticality  does not exceed the number $n$ of degrees of freedom. 	
\end{quote}
The non-positivity of the potential energy implies that the quadratic Lagrangian is strictly positive and hence  the Morse index of associated Lagrangian action functional vanished identically.
\end{rem}
Let $L \in \Lagr(n)$ and for $i=1,2$, we denote by $\nu(L_i, [0,T]) $ the total sum of all non-transversality instants (counted according their own multiplicities) between the curve $t\mapsto \ell(t)\= \psi(t)L$ and the Lagrangian subspaces $L_i\in \Lagr(n)$ on the interval $[0,T]$.
\begin{thm}{\bf [Sturm Alternation Theorem for plus-curves]\/} \label{thm:alternation-plus}
Under the above notation, the following holds:  
\begin{equation}
\Big\vert\nu(L_2,[0,T])-\nu(L_1,[0,T])\Big\vert \leq n- k
\end{equation}
where $k\= \max\{k_1,k_2\}$ and 
\[
	\begin{split}
	k_1 & \=\min\{\dim \epsilon_1,\dim \epsilon_2\} \textrm{ for } \epsilon_i\=  L \cap \ell(T) /L\cap \ell(T) \cap \L_i, i=1,2 \\
	k_2 & \=\min\{\dim \delta_1,\dim \delta_2\} \textrm{ for } \delta_1\=  L \cap L_1 +L_1 \cap L_2 \textrm{ and }\delta_2\= \ell(T) \cap L_2 + L_1 \cap L_2   .
	\end{split}
	\]
\end{thm}
\begin{proof}
		The idea of the proof is similar wit h theorem	\ref{thm:main-corollary-alternation}  but it needs more precise estimate.
		Note that $\nu(L_i,[0,T])=\iCLM (L_i,\ell(t);t\in [0,T])+\dim \ell(T)\cap L_i $ since $t\mapsto l(t)$ is $L_i$-plus curve  for $i=1,2$.
		Then we have
		\[
		\nu(L_2,[0,T])-\nu(L_1,[0,T])=s(L,\ell(T);L_1,L_2)+\dim L_2\cap \ell(T)-\dim L_1\cap \ell(T)
		\]
		Then by theorem \ref{thm:mainli}, we get
		
		\begin{equation}
		\nu(L_2,[0,T])-\nu(L_1,[0,T])=\\ \iota(L,\ell(T),L_2)+\dim L_2\cap \ell(T)-\big(\iota(L,\ell(T),L_1)+\dim L_1\cap \ell(T)\big)   \label{eq:dif_nul_1}
		\end{equation}
		\begin{equation}
		\nu(L_2,[0,T])-\nu(L_1,[0,T])=\iota(L,L_1,L_2)-\big(\iota(\ell(T),L_1,L_2)+\dim L_1\cap \ell(T) -\dim L_2\cap \ell(T) \big)\label{eq:dif_nul_2}
		\end{equation}
		By using Equation~\eqref{eq:triple} and Equation~\eqref{eq:dif_nul_1}, we get that 
		\begin{equation}
		\iota(L,\ell(T),L_i)+\dim L_i\cap \ell(t) \le n-\dim L\cap \ell(T)+\dim L\cap \ell(T)\cap L_i \label{eq:triple_estimate1}.
		\end{equation}
		Moreover, for arbitrary Lagrangian subspaces $\alpha,\beta,\gamma$, we have
		\[
		\iota(\alpha,\beta,\gamma)=n_+ Q(\alpha, \beta,\gamma)+\dim \alpha\cap \gamma-\dim \alpha\cap \beta\cap\gamma +\dim \alpha\cap \beta-\dim \alpha\cap \gamma=\iota (\beta,\gamma,\alpha).
		\] 
		Then by \eqref{eq:dif_nul_2} it follows that
		\begin{equation}
		\nu(L_2,[0,T])-\nu(L_1,[0,T]) =\iota(L,L_1,L_2)-\iota(L_1,L_2,\ell(T))\label{eq:triple_estimate2}.
		\end{equation}
		By using Equation~\eqref{eq:triple_estimate1} and  Equation~\eqref{eq:triple_estimate2}, we get the thesis  arguing precisely as given in  Theorem~\ref{thm:main-corollary-alternation} .
\end{proof}

\begin{rem}
We observe that the estimates provided in Theorem~\ref{thm:alternation-plus} is, in general, sharper than the one proved by Arnol'd for which the difference was bounded by $n$.	
\end{rem}
The next result represents a generalization of \cite[Theorem on Zeros]{Arn86}. 
\begin{thm}\label{thm:pigeon4-3}{\bf [Sturm Theorem on Zeros]\/}
Under the notation of Theorem \ref{thm:alternation-plus}, we get that  for any interval $[\alpha, \beta]\subset [0,T]$,
\begin{itemize}
\item if $\big\vert\nu\big(L_2,[\alpha,\beta])\big\vert>n-k$, then there is at least one crossing instant of $\ell$ with $L_1$;
\item if $\big\vert\nu\big(L_1,[\alpha,\beta])\big\vert>n-k$, then there is at least one crossing instant of $\ell$ with $L_2$.
\end{itemize}
\end{thm}
\begin{proof}
The proof follows immediately by using triangular inequality and Theorem \ref{thm:alternation-plus}.	
\end{proof}


\section{Sturm comparison principles} \label{sec:Optical Hamiltonians and the s-principle}

In this section we provide some new comparison principles as well as a generalization of the classical Sturm comparison principle. Our first result is a generalization of the comparison principle which was proved by third named author in \cite[Section 5]{Off00}.
\begin{thm}({\bf Comparison Principle\/})\label{thm:comparison-1}
	Let $L_1, L_2, L_3 \in \Lagr(V,\omega)$,  $\psi \in \PT{V,\omega}$ and  for $i=1,2$ we  set $\ell_i(t)\=\psi(t)L_i$. We assume that 
	\begin{enumerate}
		\item  $t\mapsto \ell_2(t)$  is $L_3$-plus curve
		\item $\itriple(L_1,L_2,L_3)=n-\dim(L_1\cap L_2)$
		\item $\iCLM(L_3, \ell_1(t); t \in [0,T])=0$.
	\end{enumerate}
Then $\iCLM(L_3, \ell_2(t);t \in [0,T])=0$.
\end{thm}
\begin{rem}
Before proving this result, we observe that assumption \textit{2.} corresponds to require that the triple index is as large as possible. In fact, by assumption \textit{1.} the term $\dim (L_1\cap L_2 \cap L_3)$ drops down. This assumption, somehow replaces the condition on $Q(L_1,L_2;L_3)$ to be positive definite  in this (maybe degenerate) situation.
\end{rem}
\begin{proof}
We start to observe that by assumption \textit{3.} $\iCLM(L_3, \ell_1(t); t \in [0,T])=0$ by assumption \textit{1.}, $\iCLM(L_3, \ell_2(t); t \in [0,T])$ is non-negative. Thus, we get
\begin{multline}
0 \leq 
\iCLM(L_3,\ell_2(t);t \in [0,T])-
\iCLM(L_3,\ell_1(t); t \in [0,T])\\
= 
\iCLM\big(\psi(t)^{-1}L_3, L_2; t \in [0,T]\big) -
\iCLM \big(\psi(t)^{-1}L_3, L_1; t \in [0,T]\big)	
= s(L_1, L_2; L_3, \psi(T)^{-1}L_3)\\=
\itriple (L_1, L_2, \psi(T)^{-1}L_3)-\itriple(L_1, L_2,L_3)\\
=\itriple (L_1, L_2, \psi(T)^{-1}L_3)-n + \dim (L_1\cap L_2) \leq 0 
\end{multline}
where the last inequality follows from Equation~\eqref{eq:triple}. In fact, 
\begin{multline*}
	\itriple (L_1, L_2, \psi(T)^{-1}L_3)\leq n- \dim(L_1\cap L_2) - \dim(L_2 \cap \psi(T)^{-1}L_3)+ \dim (L_1\cap L_2 \cap \psi(T)^{-1}L_3)\\ \leq n- \dim(L_1\cap L_2)
\end{multline*}
being $-\dim(L_2 \cap \psi(T)^{-1}L_3)+ \dim (L_1\cap L_2 \cap \psi(T)^{-1}L_3)\leq 0$. So, since  $0 \leq \iCLM(L_3,\ell_2(t); t \in [0,T])\leq 0$, we get  that $\iCLM(L_3, \ell_2(t); t \in [0,T])= 0$. This concludes the proof. 
\end{proof}
A direct consequence of the Theorem \ref{thm:comparison-1}, we get the following result which is in the form appearing in  \cite[Theorem 5.1]{Off00}.
\begin{cor}({\bf Comparison Principle\/})\label{thm:2comparison-2}
	Let $L_1, L_2, L_3 \in \Lagr(V,\omega)$,  $\psi \in \PT{V,\omega}$ and  for $i=1,2$ we  set $\ell_i(t)\=\psi(t)L_i$. We assume that 
	\begin{enumerate}
		\item $t\mapsto \ell_2(t)$  is $L_3$-plus curve
		\item $\itriple(L_1,L_2,L_3)=n-\dim(L_1\cap L_2)$
		\item $t\mapsto \ell_1(t)\in \Lagr^0(L_3)$.
	\end{enumerate}
	Then $t\mapsto \ell_2(t)\in \Lagr^0(L_3)$.
\end{cor}
\begin{proof}
		 By means of assumption \textit{1.}, we only need to prove that  $\ell_2(T)\cap L_3=\{0\}$.
		In the proof of Theorem \ref{thm:comparison-1}, we get 
		\[
		\itriple (L_1, L_2, \psi(T)^{-1}L_3)-n + \dim (L_1\cap L_2) =0 .
		\]
		Note that $\itriple (L_1, L_2, \psi(T)^{-1}L_3)\le n-\dim(L_1\cap L_2+L_2\cap \psi(T)^{-1}L_3)$.
		
		It follows that $L_2\cap \psi(T)^{-1}L_3\subset L_1\cap L_2\subset L_1$.
		Then we have
		$\psi(T)L_2\cap L_3\subset \psi(T)L_1$, and it follows that $ \psi(T)L_2\cap L_3\subset \psi(T)L_1\cap L_3=\{0\}$. 
\end{proof}	
\begin{rem}
Corollary \ref{thm:2comparison-2} provides a generalization of \cite[Theorem 5.1]{Off00} which was proved  for paths of symplectic matrices  arising as fundamental solutions of Hamiltonian systems. Moreover we removed the Legendre convexity condition as well as the transversality condition  between the Lagrangian subspaces $L_1$ and $L_2$, which, in concrete applications such a conditions are pretty difficult to be  checked.  
\end{rem}
\begin{thm}\label{thm:comparison-2}
Under the notation of Theorem \ref{thm:comparison-1}, we  assume that 
	\begin{enumerate}
		\item $t\mapsto \ell_2(t)$  is $L_3$-plus curve
		\item $\itriple(L_1,L_2,L_3)=n-\dim(L_1\cap L_2)$
		\item $\dim(L_3 \cap L_2)=k$ 
	 \item $\iCLM(L_3, \ell_1(t); t \in [a,b])=k$ for some $ k \in \N$
	\end{enumerate}
	Then $\iCLM( L_3,\ell_2(t); t \in [a,b])=k$.
\end{thm}
\begin{proof}
We start to observe that by assumption \textit{3.} and assumption \textit{1.} we get that 
\[
\iCLM(L_3,\ell_2(t);t \in [a,b])\geq k. 
\]
Thus $0 \leq \iCLM(L_3,\ell_2(t); t \in [0,T])	-\iCLM(L_3,\ell_1(t);t \in [0,T])\leq 0$ where the last inequality follows by arguing precisely as in Theorem \ref{thm:comparison-1}. By this the conclusion readily follows. 
\end{proof}
The last result of this section is a generalized version of the Sturm comparison theorem proved by Arnol'd in the case of optical Hamiltonians. The proof of this result is essentially based on spectral flow techniques and for the sake of the reader we refer to Appendix \ref{sec:spectral-flow} for the basic definitions, notation and properties.  
Now, for $i=1,2$ let us consider the Hamiltonians $H_i:[0,T] \oplus \R^{2n}\to \R$ and the induced Hamiltonian systems
\begin{equation}\label{eq:hamsys-non-2}
z'(t)=J_0 \nabla H_i\big(t,z(t)\big).
\end{equation}
By linearizing Equation \eqref{eq:hamsys-non-2} at a common equilibrium point $z_0$, we get
\begin{equation}\label{eq:hamsys-non-3}
w'(t)=J_0B_i(t)w(t),
\end{equation}
where $B_i(t) = D^2 H_i(t,z_0(t))$.
For $i=1,2$, we  denote by $\psi_i$ the fundamental solution of the corresponding linearized Hamiltonian system \eqref{eq:hamsys-non-3}.
For $s \in [0,1]$, we define the two-parameter family of symmetric matrices as follows 
\begin{equation}
	C : [0,1] \oplus [0,1] \to C^1([0,T],\Sym(2n)) \qquad 
	C_{(s,r)}(t): = C(s,r)(t) = s\big[r B_2(t)-r B_1(t)\big]+ r B_1(t).
\end{equation}
Given $L \in \Lagr(2n)$, we denote by $D(T,L)$ the subspace of $W^{1,2}$ paths defined by 
\begin{equation}\label{eq:domain}
D(T,L)\=\Set{w \in W^{1,2}([0,T], \R^{2n})|\big(w(0),w(T)\big) \in L}
\end{equation}
and we define the two parameter family of first order linear operators: 
\begin{equation}\label{eq:unbounded-oper}
	\mathcal A_{(s,r)}: D(T,L) \subset L^2([0,T], \R^n)\to 	L^2([0,T], \R^{2n}) \textrm{ defined by } \mathcal A_{(s,r)}\= -J_0 \dfrac{d}{dt} - C_{(s,r)}(t).
	\end{equation}
It is well-known that for every $(s,r)\in [0,1]\oplus[0,1]$, the linear operator $\mathcal A_{(s,r)}$ is unbounded self-adjoint in $L^2$ with dense domain $D(T,L)$. We also observe that being the domain independent on $(s,r)$ the linear operator $\mathcal A_{(s,r)}:D(T,L) \to L^2([0,T], \R^{2n})$ is bounded.  
\begin{thm}{\bf (First Comparison theorem)\/}\label{thm:main-comparison-sturm} 
Let $L \in \Lagr(2n)$ and under the notation above, we assume 
\begin{enumerate}
\item[(C1)]  $B_1(t)\leq B_2(t), \qquad \forall\, t \in [0,T]$. \end{enumerate}
Then we get
\begin{equation}
	\spfl(\mathcal A_2)\leq \spfl(\mathcal A_1)
	\end{equation} 
where $\mathcal A_1\=\mathcal A_{(0,r)}$ and $\mathcal A_2\=\mathcal A_{(1,r)}$.
\end{thm}
Before proving the result, we observe that the assumption (C1) guarantees that  the curve $s \mapsto \mathcal A_{(s,r)}$ is a plus-curve.
\begin{proof}
The proof of this result is based upon the homotopy invariance of the spectral flow. Let us consider the two parameter family of operators $ \mathcal A_{(s,r)}$ defined above, and we observe that, as direct consequence of the homotopy invariance (since the rectangle $R$ is contractible), we get that 
\begin{equation}\label{eq:nuova}
	\spfl\left( \mathcal A_{(s,0)}, s \in [0,1]\right) + \spfl\left( \mathcal A_{(1,r)}, r \in [0,1]\right)=\spfl\left( \mathcal A_{(0,r)}, r\in [0,1]\right)+\spfl\left( \mathcal A_{(s,1)}, s \in [0,1]\right).
\end{equation}
We now observe that the first term  $\spfl\left( \mathcal A_{(s,0)}, s \in [0,1]\right)=0$. This follows by the fact that $\mathcal A_{(s,0)}$ is a fixed operator. Let us now consider the second term  in the right-hand side of Equation \eqref{eq:nuova}, namely $\spfl\left( \mathcal A_{(s,1)}, s \in [0,1]\right)$. By Lemma \ref{thm:perturbazione} we can assume that for $\delta >0$ sufficiently small the path 
\[
\mathcal A^\delta_s\=\mathcal A_{(s,1)}+ \delta \,\Id 
\]
where $\Id$ denotes the identity on $L^2$, has only regular crossings.  So, by the homotopy invariance of the spectral flow we get that 
\begin{equation}\label{eq:4-22}
	\spfl\left( \mathcal A_{(s,1)}, s \in [0,1]\right)= \spfl\left( \mathcal A_s^\delta, s \in [0,1]\right)
\end{equation}
and by the assumption (C1) it follows that the local contribution to the spectral flow for the path $s\mapsto \mathcal A_s^\delta$ at each crossing instant is negative, i.e. 
\begin{equation}\label{eq:ultima-compar}
\spfl\left( \mathcal A_s^\delta, s \in [0,1]\right) \leq 0	
\end{equation}
Summing up Equation \eqref{eq:nuova}, Equation \eqref{eq:4-22} and finally Equation  \eqref{eq:ultima-compar}, we finally get that 
	\[
	\spfl\left( \mathcal A_{(1,r)}, r \in [0,1]\right) \leq\spfl\left( \mathcal A_{(0,r)}, r \in [0,1]\right).
	\]
\end{proof}
In order to relate the spectral flow for a path of Hamiltonian operators with the Maslov index of the induced Lagrangian curve, we need to use a spectral flow formula.

Let us now consider the path $s \mapsto \mathcal L_s$ of unbounded Hamiltonian operators that are selfadjoint in $L^2$ and defined on the domain $D(T,L)$ given in Equation \eqref{eq:domain}
\[
\mathcal L_s\=-J_0 \dfrac{d}{dt}- E_s(t)
\]
where $s \mapsto E_s(t)$ is a $\mathscr C^1$ path  of symmetric matrices such that $E_0(t)=0_{2n}$ and $E_1(t)=E(t)$, where  we denoted by $0_{2n}$ the $2n \oplus 2n$ zero matrix. 
\begin{prop}{\bf (Spectral flow formula)\/}\label{thm:spfl}
Under the above notation, the following equality holds
\begin{equation}
	-\spfl\left(\mathcal L_s, s \in [0,1]\right)= \iCLM(L, \Graph \psi(t);t \in [0,T])
\end{equation}
where $\psi$ denotes the solution of 
\[
\begin{cases}
	\dfrac{d}{dt}\psi(t)= J_0\, E(t) \psi(t), \qquad t \in [0,T]\\
	\psi(0)=\Id_{2n}.
\end{cases}
\]
\end{prop}
 \begin{proof}
 For the proof of this result, we refer the interested reader to \cite[Theorem 2.5, Equation (2.7) \& Equation (2.19)]{HS09}.
  \end{proof}
  \begin{rem}
   The basic idea behind the proof of Proposition \ref{thm:spfl} is to perturb the path $s\mapsto \mathcal L_s$ in order to get regular crossing  (which it is possible as consequence of the fixed endpoints homotopy invariance). Once this has been done, for concluding, it is enough to prove that the local contribution at each crossing instant to the spectral flow is the opposite of the local contribution to the Maslov index. This can be achieved by comparing the crossing forms as  in \cite[Lemma 2.4]{HS09} and to prove that the crossing instants for the path $s\mapsto \mathcal L_s$  are the same as the crossing instants of the path   $s\mapsto \mathcal \Graph \psi_s$ and at each crossing $s_0$ the kernel dimension of the operator $\mathcal L_{s_0}$ is equal to the $\dim(L\cap \Graph \psi_{s_0})$. The conclusion follows once again by using the homotopy properties of the $\iCLM$-index and the spectral flow. 
\end{rem}
\begin{thm}{\bf (Second Comparison theorem)\/}\label{thm:main-comparison-sturm-2} 
Under the notation above, we assume 
\begin{enumerate}
\item[(C1)]  $B_1(t)\leq B_2(t), \qquad \forall\, t \in [0,T]$. 
\end{enumerate}
Then we get
\begin{equation}
 \iCLM(L,\Graph \psi_1(t), t \in [0,T])\geq \iCLM(L,\Graph \psi_2(t); t \in [0,T]).
\end{equation} 
\end{thm}
\begin{proof}
The proof readily follows by Theorem \ref{thm:main-comparison-sturm} and Proposition \ref{thm:spfl}. 	
\end{proof}
As direct consequence of Theorem \ref{thm:comparison-2} we get the following useful result. 
\begin{cor}{\bf(Oscillation Theorem)\/}\label{thm:oscillation}
Let  $H: [0,T] \oplus \R^{2n}\to \R$ be a $\mathscr C^1$ natural quadratic Hamiltonian of the form 
\[
H(t,p,q)= \dfrac12\norm{p}^2 + V(t,q), \qquad (t,q,p) \in [0,T]\times \R^{2n}
\] 
such that 
\[
V(t,q) \leq \dfrac12 \omega^2\norm{q}^2 \textrm { and } V(0,q)= \dfrac12 \omega^2\norm{q}^2
\]
 Then, we get 
 \[
 \iCZ(\psi(t); t \in [0,T])\geq 2\left\lfloor{\dfrac{T\omega}{2\pi}}\right\rfloor.
\]
In particular, this number growth unboundedly as $T\to +\infty$.
 \end{cor}
\begin{proof}
The proof follows as direct application of Theorem \ref{thm:main-comparison-sturm-2}, in the case in which $L=\Delta$ and of \cite[Equation (3.8)]{KOP19}. 
\end{proof}
\begin{rem}
An analogous of Corollary \ref{thm:oscillation} already appears in \cite[Corollary 2 (Oscillation Theorem]{Arn86}. In this result, however, author estimates from below the moments of verticality, namely the Maslov index with respect to the Dirichlet Lagrangian. We also observe that the opposite inequality appearing in  Corollary \ref{thm:oscillation} with respect to the aforementioned  Arnol'd result is due essentially to the fact that in that paper author considered Lagrangian paths ending in the vertex of the train, whereas we are considering Lagrangian paths starting at the vertex of the train. 
\end{rem}
We close this section with a comparison theorem for Morse-Sturm systems. For $i=1,2$, let us consider the natural quadratic Hamiltonians $H_i:\R^{2n}\to \R$ of the form 
\begin{equation}
H_i(p,q)= \dfrac12\langle P_i(t)^{-1}p,p\rangle-\dfrac12\langle R_i(t)q,q\rangle 
\end{equation}
where $t\mapsto P_i(t)$ and$t\mapsto R_i(t)$ are $\mathscr C^1$-paths  symmetric matrices and $P_i(t)$ is positive definite for all $t \in [0,T]$. Thus the Hamiltonian system given in Equation \eqref{eq:hamsys-finale} reduces to 
\begin{equation}\label{eq:hamsys-finale2}
z_i'(t)=J_0 B_i(t)\, z(t), \qquad t \in [0,T] \qquad \textrm{ where }\qquad
B_i(t)\=\begin{bmatrix}
	P_i^{-1}(t) & 0\\
	0 & -R_i(t)
\end{bmatrix}. 
\end{equation}
Let $Z\subset \R^{n}\oplus \R^n$ be a linear subspace, $L_Z \in \Lagr(2n)$ be the Lagrangian subspace defined by Equation \eqref{eq:utile-p-focali} and, for $i=1,2$, we denote by $\iMorse{Z}(B_i)$ the Morse-index of the index form of the Morse-Sturm system corresponding to $B_i$. 
\begin{prop}\label{thm:comparison-MS}
Under the above notation, we assume that 
\begin{enumerate}
\item[(S1)] $P_1(t)^{-1}\leq P_2(t)^{-1}$ for every $t \in [0,T]$;
\item[(S2)] 	$R_1(t) \geq R_2(t)$ for every $t \in [0,T]$;
\end{enumerate}
	Then we get 
	\[
	\iMorse{Z}(B_1) \geq \iMorse{Z}(B_2).
	\]
\end{prop}
\begin{proof}
Under (S1) \& (S2), it follows that $B_1(t) \leq B_2(t)$ for all $t \in [0,T]$. Thus as direct consequence of Theorem \ref{thm:main-comparison-sturm-2}, we get 
\begin{equation}
\iCLM(L_Z,\Graph \psi_1(t); t \in [0,T])
\geq
\iCLM(L_Z,\Graph \psi_2(t); t \in [0,T]).
\end{equation}
By Theorem \ref{thm:index} we infer that 
$\iCLM(L_Z,\Graph \psi_i(t); t \in [0,T]) = \iMorse{Z}(B_i)+C(Z)$ and so the thesis follows. This concludes the proof. 
\end{proof}

%

\section{Some applications in geometry and  classical mechanics}

The aim of this final section is to give some applications in differential geometry and  in classical mechanics. Inspired by \cite{JP09} from which we borrow some notation, in Subsection \ref{subsec:semirie} we shall prove some comparison results between the conjugate and focal points along a geodesic on semi-Riemannian manifold. In Subsection \ref{subsec:simple-mechanical-systems} some applications to the planar Kepler problem where provided.

%

\subsection{Comparison Theorems in semi-Riemannian geometry}\label{subsec:semirie}

Let $(M,g)$ be semi-Riemannian $n$-dimensional manifold, and let $D$ be the covariant derivative of the Levi-Civita connection of the metric tensor $g$. We denote by $R$ the Riemannian curvature tensor, chosen according to the following sign convention $R(\xi,\eta)\=[D_\xi, D_\eta]-D_{[\xi,\eta]}$. Given a geodesic $\gamma:[a,b] \to M$ the {\em Jacobi (deviation) equation along $\gamma$\/} is given by 
\begin{equation}\label{eq:Jacobi-equation}
(D/dt)^2 \xi(t)-R\big(\gamma'(t), \xi(t)\big)\gamma'(t)=0\qquad \forall\, t \in [a,b].
\end{equation}
The Jacobi equation is a linear second order differential equation whose flow $\Phi$ defines a family of isomorphisms 
\[
\Phi_t: T_{\gamma(a)}M \oplus T_{\gamma(a)}M \to T_{\gamma(t)}M \oplus T_{\gamma(t)}M \qquad \textrm{ for } t \in [a,b]
\]
 defined by $\Phi_t(v,w)\= \big(J_{v,w}(t),  (D/dt)J_{v,w}(t)\big)$ where $J_{v,w}$ is the unique Jacobi field along $\gamma$ satisfying $J(a)=v$ and $(D/dt)J(a)=w$.
 
  On the space $V\=  T_{\gamma(a)}M \oplus T_{\gamma(a)}M$, let us consider the symplectic form given by 
 \[
 \omega\big((v_1, w_1), (v_2, w_2)\big)\=g(v_1, w_2)- g(v_2, w_1)
 \]
and for all $t \in [a,b]$ we define $L_0^t=\{0\}\oplus T_{\gamma(t)}M \subset V$ and we set $\ell(t)\=\Phi_t^{-1}(L_0^t)$. It is easy to check that in this way we get a smooth curve $\ell:[a,b] \to \Lagr(V,\omega)$. We set $L_0\=\ell(a)=L_0^a$.\footnote{
We observe that even if the local chart of the atlas of the Lagrangian Grassmannian manifold is the opposite with respect to that one defined by authors in \cite{JP09}, there is no sign changing involved, since our symplectic form is the opposite of the symplectic form defined by authors in the aforementioned paper and the two minus signs cancel each other. 
}
Now, consider  a smooth connected submanifold $P$ of $M$, with $\gamma(a) \in P$ and $\gamma'(a) \in T_{\gamma(a)}P^\perp$ (where $\perp$ is the orthogonal  with respect to $g$) and we assume that the restriction of $g$ to $T_{\gamma(a)}P$ is non-degenerate. (This condition is always true if $M$ is either Riemannian or Lorentzian and $\gamma$ is timelike). Let $S$ be the second fundamental form of $P$ at $\gamma(a)$ in the normal direction $\gamma'(a)$, seen as a $g$-symmetric operator $S: T_{\gamma(a)}P \to T_{\gamma(a)}P$. 
\begin{defn}
A {\em  $P$-Jacobi field\/} is a solution $\xi$ of Equation \eqref{eq:Jacobi-equation} such that $\xi(a) \in T_{\gamma(a)}P$ and $(D/dt)\xi(a)+ S[\xi(a)]\in T_{\gamma(a)}P^\perp$. 
\end{defn}
An instant $t_0\in(a,b]$ is {\em $P$-focal\/} if there exists a nonzero $P$-Jacobi field vanishing at $t_0$. The multiplicity of a mechanical $P$-focal instant is the multiplicity of the $P$-Jacobi fields vanishing at $t_0$. To every submanifold $P$ of $M$, we associate a Lagrangian subspace $L_P\subset V$ defined by 
\begin{equation}\label{eq:concrete-lagrangian}
	L_P\=\Set{(v,w)\in T_{\gamma(a)}M\oplus T_{\gamma(a)}M| v \in T_{\gamma(a)}P \textrm{ and }w+S(v) \in T_{\gamma(a)}P^\perp}.
\end{equation}
It is worth noticing that, if the submanifold $P$ reduces to the point $\gamma(a)$, then the induced Lagrangian defined in Equation  \eqref{eq:concrete-lagrangian} reduces to  $L_0\=T_{\gamma(a)}M\oplus \{0\}$ and we term a  $P$-focal point just a {\em conjugate point\/}. Then, an instant $t \in ]a,b]$ is  $P$-focal along $\gamma$ if and only if $\ell(t)\cap L_P\neq \{0\}$ and the dimension of the intersection coincides with the multiplicity of the  $P$-focal point. We also observe that $L_0\cap L_P= T_{\gamma(a)}P^\perp \oplus \{0\}$ and hence $\dim (L_0\cap L_P)= \codim P$

For all $t \in ]a,b]$, we define the space 
\begin{equation}
A_P[t]\=\Set{(D/dt) J(t)| J \textrm { is a $P$-Jacobi field along }	\gamma \textrm{ such that } J(t)=0},
\end{equation}
whilst for $t=a$ we set $A_P[a]= T_{\gamma(a)}P^\perp$. We observe that $\dim A_P[t] = \dim \ell(t)\cap L_P$. If $P$ is just a point for all $t \in ]a,b]$, we set 
\begin{equation}
A_0[t]\=\Set{(D/dt) J(t)| J \textrm { is a $P$-Jacobi field along }	\gamma \textrm{ such that } J(a)=J(t)=0},
\end{equation}
whilst for $t=a$ we set $A_0[a]= T_{\gamma(a)}M$. As direct application of Theorem \ref{thm:main-corollary-alternation}, we get the following comparison between conjugate and focal points. 

\begin{prop}\label{thm:conj-focal}Under the previous notation, the following inequality holds
\begin{equation}
\Big \vert
\iCLM\big(L_P, \ell(t);t \in [a,b]\big)-
\iCLM\big(L_0, \ell(t);t \in [a,b]\big)
\Big\vert 
\leq n- k 
\leq \dim P
\end{equation}
where  $k= \dim \big(\ell(b) \cap L_0 + L_0 \cap L_P)$.
\end{prop}
\begin{rem}
The last inequality appearing in Proposition  \ref{thm:conj-focal} coincide with that one proved by authors in \cite[Proposition 4.3]{JP09}. 
\end{rem}
As direct consequence of the triangular inequality and Proposition \ref{thm:conj-focal}, we get the following. 
\begin{cor}
Under the notation of Proposition \ref{thm:conj-focal}, we get that, for any interval $[\alpha, \beta]\subset [a,b]$,
\begin{itemize}
	\item if $\iCLM\big(L_0, \ell(t);t \in [a,b]\big)>n-k$ then there is at least one mechanical $P$-focal instant in $[\alpha,\beta]$
	\item if $\iCLM\big(L_P, \ell(t);t \in [a,b]\big)>n-k$ then there is at least one mechanical $P$-conjugate instant in $[\alpha,\beta]$
\end{itemize}
\end{cor}
The last result of this paragraph is quite useful in the applications. Loosely speaking, claims that the absence of conjugate (respectively focal instants gives an upper bound on the number of focal (respectively conjugate) instants
\begin{prop}
	If $\gamma$ has no conjugate instant, then 
	\[
	|\iCLM(L_P, \ell(t); t \in [\alpha, \beta])| \leq n-k 
	\]
	for   $k= \dim \big(\ell(b) \cap L_0 + L_0 \cap L_P)$ and for every $[\alpha, \beta] \subset ]a,b]$. Similarly, if $\gamma$ has no $P$-focal instants, then 
	\[
	|\iCLM(L_0, \ell(t);t \in [\alpha, \beta])| \leq n-k. 
	\]
\end{prop}
\begin{proof}
If $\gamma$ has no conjugate instants, then $\iCLM\big(L_0, \ell(t);t \in [a,b]\big)=0$. The result directly follows by applying Proposition\ref{thm:conj-focal}. Similarly for the second claim.  
\end{proof}
Let now consider  two smooth connected submanifold $P,Q$ of $M$, with $\gamma(a) \in P \cap Q$ and $\gamma'(a) \in T_{\gamma(a)}P^\perp\cap T_{\gamma(a)}Q^\perp$ (where $\perp$ is the orthogonal  with respect to $g$) and we assume that the restriction of $g$ to $T_{\gamma(a)}P$ and to $T_{\gamma(a)}Q$ are non-degenerate.  
We set 
\begin{multline}
	L_P\=\Set{(v,w)\in T_{\gamma(a)}M\oplus T_{\gamma(a)}M| v \in T_{\gamma(a)}P \textrm{ and }w+S_P(v) \in T_{\gamma(a)}P^\perp}\\
	L_Q\=\Set{(v,w)\in T_{\gamma(a)}M\oplus T_{\gamma(a)}M| v \in T_{\gamma(a)}Q \textrm{ and }w+S_Q(v) \in T_{\gamma(a)}Q^\perp}
\end{multline}
where $S_P$ and $S_Q$ denote the shape operators of $P$ and $Q$, respectively.
\begin{prop}
Let $L$ be either $L_P$ or $L_Q$. Then we have 
\begin{equation}
\Big \vert\iCLM\big(L, \ell(t);t \in [a,b]\big)-\iCLM\big(L_0, \ell(t);t \in [a,b]\big)\Big\vert \leq n- k\leq d
\end{equation}
where  $k=\max\{k_P, k_Q\}$ for
\[
k_P= \dim \big(\ell(b) \cap L_0 + L_0 \cap L_P) \textrm { and }k_Q= \dim \big(\ell(b) \cap L_0 + L_0 \cap L_Q)
\]
and  $d\=\max\{\dim P, \dim Q\}$.
\end{prop}


\subsection{Simple Mechanical systems and mechanical focal points}\label{subsec:simple-mechanical-systems}

This final section is devoted to study  the so-called $P$-kinetic focal and conjugate points in the case of simple mechanical systems  and to derive some interesting estimates relating the qualitative and variational  behavior of orbits in some singular Lagrangian systems.

In this paragraph we stall by recalling some well-known facts and to fix our notation. The main references are \cite{Sma70a, Sma70b, Pin75} and references therein. \\

\begin{defn}
Let $(M,g)$ be a finite dimensional Riemannian manifold and $V:M \to \R$ be a smooth function. The triple $(M,g,V)$ is called a {\em simple mechanical system\/}. The manifold $M$ is 	called the {\em configuration space\/} and its tangent bundle $TM$ is usually called the {\em state space\/}. A point in $TM$ is a {\em state\/} of the mechanical system which gives the position and the velocity. The {\em kinetic energy\/} $K$ of the simple mechanical system is the function 
\begin{equation}\label{eq:kinetic}
K: TM \to \R \textrm{ defined by } K(q, v)\=\dfrac12 \norm{v}_g^2\quad \forall\, (q,v) \in TM.
\end{equation}
The smooth function $V$ is called the {\em potential energy (function)\/} of the system  and finally the {\em total energy function\/}
\begin{equation}\label{eq:total-energy}
E: TM \to \R \textrm{ defined by } E(q, v)\=\dfrac12 \norm{v}_g^2+ V(q)\quad \forall\, (q,v) \in TM.
\end{equation}
\end{defn}
\begin{note}
Everywhere in the paper we shall  denote 	 by $V$ the potential energy and by $U$ the potential function and we recall that $V=-U$.
\end{note}

\begin{ex}{\bf (The $n$-body problem)\/}
	Consider $n$ point masses particles ({\em bodies\/}) with masses $m_1, \ldots ,m_n \in \R^+$  moving in the $d$-dimensional Euclidean space $E^d$. So the positions of the bodies is described by the vector $q=(q_1, \ldots, q_n)\in (E^d)^n$. The kinetic energy is 
	\begin{equation}\label{eq:kinetic-n-bodies}
	K(q,v)\=\dfrac12\sum_{i=1}^n \langle m_i v_i, v_i\rangle \qquad \forall \,(q,v) \in (E^d)^n \times (E^d)^n.
	\end{equation}
	Clearly the kinetic energy is induced by the Riemannian metric $\langle cdot, \cdot \rangle_M $ on $(E^d)^n$ defined by 
	\[
	\langle v, w\rangle_M = \sum_{i=1}^n \langle m_i v_i,w_i \rangle \qquad \forall\, v,w \in (E^d)^n.
	\]
	The $n$-bodies moves under the influence of the Newtonian potential energy defined by 
	\begin{equation}\label{eq:Newtonial-potential-function}
		V(q_1, \dots q_n)=-\sum_{i<j}\dfrac{m_im_j}{\norm{q_i-q_j}}.
	\end{equation}
The function $V$ is singular at the {\em collision set\/} defined by 
\begin{equation}
\Delta\=\Set{(q_1, \ldots, q_n)\in (E^d)^n| q_i=q_j \textrm{ for some } i \neq j}.
\end{equation}
Then $V$ is a smooth function on $M\= (E^d)^n\setminus \Delta$ thus defining a simple dynamical system $(M,K,V)$.
\end{ex}
\begin{defn}
	A {\em physical path (orbit, trajaectory)\/} of a simple mechanical system $(M,g,V)$ is a smooth path $\gamma$ in $M$ satisfying the Newton Equation 
	\begin{equation}\label{eq:Newton}
		(D/dt)\gamma'= -\nabla_g V(\gamma)
	\end{equation}
	where $D/dt$ denotes the covariant derivative relative of the Levi-Civita connection $D$ of the Riemannian metric $g$ and where $\nabla_g$ denotes the gradient defined by $g$.
\end{defn}
\begin{rem}
If $V=0$ then the physical path are just geodesics of the Riemannian manifold.	Moreover if $g$ is the Euclidean metric, then the left-hand side of Equation \eqref{eq:Newton} reduces to $\gamma''$ and the gradient $\nabla_g$ appearing in the right-hands side of that equation is the usual gradient. 
\end{rem}
By the conservation law of the total energy function along a physical path and since in the Riemannian world the kinetic energy is non-negative\footnote{
This fact is not longer true, in general,  on semi-Riemannian manifolds having non trivial signature (for instance Lorentzian manifolds).} a physical path of total energy $h \in \R$ must lie in the set 
\begin{equation}
	\overline M\=\Set{q \in M| V(q) \leq h},
\end{equation}
where $\overline M$ denotes the topological closure of the set 
\begin{equation}
	 M\=\Set{q \in M| V(q) < h}
\end{equation}
usually called the  {\em $h$-configuration space\/} or the {\em Hill's region\/}. If $h$ is a regular value of $V$, then $\overline M$ is a smooth manifold with boundary 
\begin{equation}\label{eq:boundary-hill}
\partial M\=\Set{q \in M| V(q)=h}. 
\end{equation}
The {\em Jacobi metric\/} $g$ corresponding to the value $h$ of a simple mechanical system $(M,g,V)$ is given by 
\begin{equation}\label{eq:Jacobi-metric}
g(q)\= 2[h-V(q)] \, g(q).	
\end{equation}	
\begin{rem}
	We observe that $g$ defines a honest Riemannian metric on $M$ which degenerate on $\partial M$. 
\end{rem}
The next result, which relates the physical paths of energy $h$ and the geodesics on the Hill's region with respect to the Jacobi metric, goes back to Jacobi. 
\begin{prop}{\bf (Jacobi)\/}
The physical paths of $(M,g,V)$ of total energy $h$ are, up to time re-parametrization, geodesics of the Riemannian manifold $(M, g)$.	
\end{prop}
We now consider  the configuration space $M$ to be the Euclidean plane $E^2$ endowed with a polar coordinate system $(r,\theta)$. Take the origin to be the center of central force so that the potential energy $V$ of the problem depends only upon $r$ (thus is $\theta$ independent). We assume that the particle has mass $m=1$ so that the kinetic energy is $K(q,v)=\norm{v}^2/2$ for all $v \in E^2$. The Jacobi metric of this simple mechanical system in polar coordinates is  given by 
\[
g\= 2[h-V(r)](dr^2+ r^2 d\theta^2). 
\]
The {\em mechanical Gaussian curvature\/} can be easily computed (cfr. \cite[Proposition 2.1]{Pin75}) and it is given by
\begin{equation}
\mathcal K(q)\= \dfrac{1}{4[h-V(r)]}\big[(h-V)(rV')'+r(V')^2\big].
\end{equation}
Assuming that $h$ is a regular value of $V$ meaning that $V'\neq 0$ on the boundary ring 
\[
\partial M\=\Set{q \in M| V({ \norm{q}})=h}\neq \emptyset,
\]
then by continuity it readily follows  the following result. 
\begin{lem}{\bf \cite[Proposition 2.1 \& Proposition 2.2]{Pin75}\/}
Suppose $h$ is a regular value of $V$ and that the boundary ring $\partial M \neq \emptyset$. Then there is an annulus region of the boundary $\partial M$ on which the mechanical	 Gaussian curvature is positive. Moreover $\mathcal K(q) \to +\infty$ as $q \to \partial M $.  
\end{lem}
\subsubsection*{The planar Kepler problem}
In polar coordinates the Jacobi metric  for the planar Kepler problem is 
\begin{equation}
	g= 2\left(h+\dfrac{1}{r}\right)(dr^2+r^2d\theta^2).
\end{equation}
\begin{rem}
As recently observed by Montgomery in \cite[Section 4]{Mon18}, in the particular case of zero energy $h=0$ it reduces to 
\[
g_0=2\left(\dfrac{dr^2}{r}+d\theta^2\right)
\]
and by setting $\rho=2r^{1/2}$ it can be written  as follows 
\[
g_0= d\rho^2+ \dfrac{\rho^2}{4}d\theta^2
\]
which is the metric of {\em cone over a circle of radius $1/2$\/}. 
\end{rem}
In the standard planar Kepler problem, the mechanical Gaussian curvature is 
\begin{equation}
\mathcal K(r)= -\dfrac{h}{4(1+rh)^3}.
\end{equation}
In particular we get that 
\[
\begin{cases}
h>0 \quad \Rightarrow \quad \mathcal K(r)<0 &\textrm{{\bf (hyperbolic orbits)\/} }\\
h=0 \quad \Rightarrow \quad \mathcal K(r)=0 &\textrm{{\bf (parabolic orbits)\/} } \\
 h<0 \quad \Rightarrow \quad \mathcal K(r)>0&\textrm{{\bf (elliptic orbits)\/} }.	
\end{cases}
\]
In the two dimensional case the mechanical Jacobi field, reduces to 
\begin{equation}\label{eq:kepler}
	\dfrac{d^2 J}{ds^2}+ \mathcal K(s) J=0
\end{equation}
where $s$ denotes the Jacobi arc-length.  Since $|\mathcal K|\geq |h|/4$, and as a direct consequence of Proposition~\ref{thm:comparison-MS}, we get the following. 
\begin{thm}
Let $\gamma$ be a Keplerian ellipse. Then the first  conjugate point occurs at Jacobi distance less than 
\[
2\dfrac{\pi}{\sqrt{|h|}}.
\]
\end{thm}
\begin{proof}
In fact, since $|\mathcal K(s)|	\geq \dfrac{|h|}{4}$, by setting $R_1(s)=|\mathcal K(s)|$ and  $R_2(s)\= |h|\Id$ and by  using Proposition~\ref{thm:comparison-MS},  we get that the associated block diagonal  matrices $B_1$ and $B_2$ are ordered, meaning that pointwise we have $B_1(s) \leq B_2(s)$ for every  $s \in [0,1]$. Thus, by invoking once again Proposition~\ref{thm:comparison-MS} and Theorem~\ref{thm:index}, we have
\[
	\iMorse{L_D}(B_1) \geq \iMorse{L_D}(B_2).
	\]
	Since  crossing instants (or a verticality moments) correspond to conjugate points. (Cfr. \cite{MPP05} and references therein for further details), the result follows once observed that 
 $|\mathcal K|\geq |h|/4$ and $|h|/4$ is the Gaussian curvature of the sphere of radius $2/\sqrt{|h|}$. This concludes the proof. 
\end{proof}


\appendix

\section{A symplectic excursion on the Maslov index}\label{sec:Maslov}

The purpose of this Section is to provide the basic definitions, properties and symplectic  preliminaries used in the paper. We recall the basic definition, the main 
properties  of the intersection number for curves of Lagrangian subspaces with respect to a  distinguished one and we fix our notation.   Our basic references are \cite{RS93, CLM94, LZ00, MPP05,MPP07, HS09, BJP14, BJP16, PWY19}.

\subsection{Symplectic preliminaries and the Lagrangian Grassmannian}\label{subsec:preliminaries}

A finite dimensional (real) {\em symplectic vector space\/}, is a pair  $(V, \omega)$, where $V$ is a (real, even dimensional) 
vector space, and $\omega: V \times V \to \R$ is an antisymmetric non-degenerate bilinear form on $V$. A {\em complex structure\/} on the real vector space $V$ is an automorphism  $J: V \to V$ such that  $J^2=-\Id$. With such a structure $V$ becomes a complex vector space. 
We denote by  $\Sp(V, \omega)$ the {\em symplectic group \/} of  $(V,\omega)$ which is the  closed Lie subgroup of the general linear group $\GL(V)$ consisting of all isomorphisms that preserve $\omega$. The  Lie algebra $\ssp(V,\omega)$ of $\Sp(V, \omega)$ consists of all endomorphisms $X: V \to V$ such that $\omega(X\cdot, \cdot)$  is a symmetric bilinear form on $V$, i.e. $ \omega(Xv, w)=\omega(Xw,v)$, for all    $v, w\, \in V$. 
 Here and throughout, unless different stated, $(V,\omega)$ denotes a $2n$-dimensional (real) symplectic space.
 
We start by recalling some classical definition and notation that we will use throughout the paper. First of all, a (linear) subspace   $I \subset V$ is termed {\em isotropic\/} if the restriction of $\omega$ on $I$ vanishes identically. Now, given an  isotropic subspace $I$ of the symplectic Euclidean space $(V, \langle \cdot, \cdot \rangle, \omega)$, 
we shall identify the quotient space $I^\omega/I$ with the orthogonal complement $V_I$ of $I$ in $I^\omega$ and we 
call $V_I$ the {\em symplectic reduction of $V$ modulo $I$.\/} 
Thus, by definition:
\begin{equation}\label{3.1}
V_I\,:=\, I^\omega \cap I^\perp\,=\, (J I)^{\perp }\cap I^\perp 
\end{equation}
Notice that if $I$ is isotropic, also $JI$ is isotropic. Moreover  $V_I \, = \, V_{JI}.$  This follows from Equation~\eqref{3.1} 
and the orthogonality relations between
$\omega $ and $\perp.$ Moreover 
\begin{equation}\label{eq:ortogonalisimplettici}
V_I^{\perp}\,=\,[I^{\perp}\cap (JI)^{\perp}]
^{\perp}\,=\, I\oplus JI.
\end{equation}
We observe that $V_I$ is a symplectic space since  $V_I\cap V_I^\omega=\{0\}$.
Thus, we get the symplectic decomposition of $V$:
$ V=V_I \oplus V_I^\perp.$
A special class of isotropic subspaces is played by the so-called {\em Lagrangian subspaces.\/} More precisely, a  maximal (with respect to the inclusion) isotropic subspace of $(V,\omega)$ is termed a {\em Lagrangian subspace\/}. We denote by $\Lagr(V, \omega)$  (or in shorthand notation  by $\Lagr$) the  collection of all Lagrangian subspaces of $V$. So, if  $(V,\omega)$ is a $2n$-dimensional (real) symplectic space,  a  {\em Lagrangian  subspace\/} of $V$ is an $n$-dimensional subspace $L \subset V$ such that $L =  L^\omega$ where $L^\omega$ denotes the {\em symplectic orthogonal\/}. We denote by $ \Lagr= \Lagr(V,\omega)$ the {\em Lagrangian Grassmannian of 
$(V,\omega)$\/}, namely the set of all Lagrangian subspaces of $(V, \omega)$; thus $\Lagr(V,\omega)\=\Set{L \subset V| L= L^{\omega}}.$
\begin{note}
Here and throughout the Lagrangian Grassmannian of the standard symplectic space will be denoted by $\Lagr(n)$. Moreover, we set  
\begin{equation}
L_D=  \R^n\times \{0\} \subset \R^n \times \R^n
\quad \textrm{ and } \quad 
L_N= \{0\} \times \R^n \subset \R^n \times \R^n
\end{equation}
and we shall refer to $L_D$ as the {\em Dirichlet (or horizontal) Lagrangian subspace\/} whilst  to $L_N$ as the {\em Neumann (vertical) Lagrangian subspace.\/}
\end{note}

%

 In this subsection we recall some basic facts on the differentiable structure of $\Lagr(V,\omega)$. We start to observe that  $\Lagr(V,\omega)$ has the structure of a compact real-analytic submanifold of the Grassmannian of all $n$-dimensional subspaces of $V$. Moreover the dimension of $\Lagr(V,\omega)$  is $\frac12 n(n+1)$ and an atlas on $\Lagr$ is given as follows.
 
Given a   {\em Lagrangian decomposition\/} of $(V, \omega)$ namely a pair $(L_0, L_1)$ of Lagrangian subspaces of $V$ 
with $V=L_0 \oplus L_1$,  we denote by $\Lagr^0(L_1)$ the open and dense subset of $\Lagr(V, \omega)$ consisting  of all Lagrangian  subspaces of $V$ that are transversal to $L_1$. To any    Lagrangian decomposition $(L_0, L_1)$ of $V$ it remains a well-defined  bijection 
\begin{equation}\label{eq:carte}
  Q(L_0,L_1):\Lagr^0(L_1)\to \Bsym(L_0) \ \textrm{ defined  by } \  
Q(L_0,L_1)(L)\=Q(L_0,L_1;L)\= \omega(\cdot, T\cdot)\big\vert_{L_0 \times L_0}
\end{equation}
where $T: L_0 \to L_1$ is the unique linear map whose graph in $V$ is represented by $L$.\footnote{We observe that this map coincides, up to a sign with, the one defined in \cite[Equation 2.3]{Dui76} or with the local chart $\varphi_{L_0,L_1}(L)$ given by authors in \cite[Section 2]{DDP08} or in \cite[Section 2]{JP09}. 
However our choice is coherent with the crossing forms defined through $Q$  in  \cite[Section 1]{RS93}, \cite[Equation 2 \& Remark 3.1]{ZWZ18} with \cite[Section 3]{LZ00} and \cite{CLM94}.
}
 We also observe that $ \ker \big(Q(L_0,L_1;L)\big)= L \cap L_0,$ for all $L \in \Lambda^0(L_1)$.  Moreover, as proved by the author in \cite[Proposition 2.1]{Dui76}, the collection of all $Q(L_0,L_1)$ where the pair $(L_0, L_1)$ runs all over the Lagrangian decomposition of $(V,\omega)$ form a differentiable atlas for $\Lagr(V,\omega)$.
 For any distinguished  $L_0 \in \Lagr$, 
let $ \Lagr^k(L_0) \= \Set{L \in \Lagr(V,\omega) | \dim\big(L \cap L_0\big) =k } \qquad k=0,\dots,n.$
We recall that $\Lagr^k(L_0)$ is a real compact, connected submanifold of codimension  $k(k+1)/2$. The topological closure 
of $\Lagr^1(L_0)$  is the {\em Maslov cycle\/} that can be also described as follows. 
\begin{defn}\label{def:universal-maslov-cycle}
We term {\em Maslov cycle with vertex at $L_0$\/} or  {\em train with vertex $L_0$\/}(by using Arnol'd terminology \cite[Section 2]{Arn86}),  the algebraic (stratified) variety defined by 
\begin{equation}\label{eq:univ-Maslov-cycle}
\Sigma(L_0) \= \bigcup_{k=1}^n \Lagr^k(L_0).
\end{equation}
\end{defn}
The top-stratum $\Lagr^1(L_0)$ is co-oriented meaning that it has a 
transverse orientation. To be 
more precise, for each $L \in \Lagr^1(L_0)$, the path of Lagrangian subspaces 
$(-\delta, \delta) \mapsto e^{tJ} L$ cross $\Lagr^1(L_0)$ transversally, and as 
$t$ increases the path points to the transverse direction. Thus  the Maslov cycle is two-sidedly embedded in 
$\Lagr(V,\omega)$ and,  based on the topological properties of the Lagrangian  Grassmannian manifold,  it is possible to define a fixed endpoints homotopy invariant $\iCLM$-which is a generalization of the classical notion of  {\em Maslov index\/} for paths of Lagrangian subspaces.


\subsection{On the CLM-index: definition and computation}

Our basic references for this subsection are the beautiful papers \cite{RS93, CLM94, LZ00}.

We let $\mathscr P([a,b]; \R^{2n})$ the space of continuous maps 
\[
f: [a,b] \to \Set{\textrm{pairs of Lagrangian subspaces in } \R^{2n}}
\]
equipped with the compact-open topology and we recall the following definition. 
\begin{defn}\label{def:Maslov-index}
The {\em CLM-index\/} is the unique integer valued function  
\[
 \iCLM: \mathscr P([a,b]; \R^{2n}) \to \Z
\]
which satisfies Properties I-(VI) in \cite{CLM94}.
\end{defn}

For further reference we refer the interested reader to \cite{CLM94} and references therein.  Following authors in \cite[Section 3]{LZ00}, and references therein, let us now introduce the notion of crossing form that gives an efficient  way for computing the intersection indices   in the Lagrangian Grassmannian context.  

Let $\ell$ be a $\mathscr C^1$-curve of Lagrangian subspaces 
such that 
$\ell(0)= L$ and $\dot \ell(0)=\widehat L$. Now, if  $W$ is a fixed Lagrangian subspace transversal to $L$. For  $v \in L$ and  small enough $t$, let $w(t) \in W$ be such that $v+w(t) \in \ell(t)$.  Then the  form 
\begin{equation}\label{eq:forma-Q}
 Q(L, \widehat L)[v]= \dfrac{d}{dt}\Big\vert_{t=0} \omega \big(v, w(t)\big)
\end{equation}
is independent on the choice of $W$. 
\begin{defn}\label{def:crossing-form}
Let $t \mapsto \ell(t)=(\ell_1(t), \ell_2(t))$ be a map in  	 $\mathscr P([a,b]; \R^{2n})$. For $t \in [a,b]$, the crossing form is a quadratic form defined by 
\begin{equation}\label{eq:crossings}
\Gamma(\ell_1, \ell_2, t)= Q(\ell_1(t), \dot \ell_1(t))- 	Q(\ell_2(t), \dot \ell_2(t))\Big\vert_{\ell_1(t)\cap \ell_2(t)}
\end{equation}
 A {\em crossing instant\/} for the curve $t \mapsto \ell(t)$ is an instant $t \in [a,b]$  such that $\ell_1(t)\cap \ell_2(t)\neq \{0\}$ nontrivially. A crossing is termed {\em regular\/} if the $\Gamma(\ell_1, \ell_2, t)$ is non-degenerate. 
 \end{defn}
 We observe that  if $t$ is a crossing instant, then $
 \Gamma(\ell_1, \ell_2,t)= - \Gamma (\ell_2, \ell_1, t).$
 If  $\ell$ is {\em regular\/} meaning that  it has only regular crossings, then the $\iCLM$-index can be computed through the crossing forms, as follows 
\begin{equation}\label{eq:iclm-crossings}
 \iCLM\big(\ell_1(t), \ell_2(t); t \in [a,b]\big) = \coiMor\big(\Gamma(\ell_2, \ell_1, a)\big)+ 
\sum_{a<t<b} 
 \sgn\big(\Gamma(\ell_2, \ell_1, t)\big)- \iMor\big(\Gamma(\ell_2, \ell_1,b)\big)
\end{equation}
where the summation runs over all crossings $t \in (a,b)$ and $\coiMor, \iMor$  are the dimensions  of  the positive and negative spectral spaces, respectively and $\sgn\= 
\coiMor-\iMor$ is the  signature. 
(We refer the interested reader to \cite{LZ00} and \cite[Equation (2.15)]{HS09}). 

Let $L_0$ be a distinguished Lagrangian and we assume that $\ell_1(t)\equiv L_0$ for every $t \in [a,b]$. In this case we get that the crossing form at the instant $t$ provided in Equation~\eqref{eq:crossings} actually reduce to 
\begin{equation}\label{eq:forma-crossing}
 \Gamma\big(\ell_2(t), L_0, t \big)= Q|_{\ell_2(t)\cap L_0}
\end{equation}
and hence   
\begin{equation}\label{eq:iclm-crossings-2}
 \iCLM\big(L_0, \ell_2(t); t \in [a,b]\big) = \coiMor\big(\Gamma(\ell_2, L_0, a)\big)+ 
\sum_{a<t<b} 
 \sgn\big(\Gamma(\ell_2, L_0, t)\big)- \iMor\big(\Gamma(\ell_2, L_0,b)\big)
\end{equation}
\begin{rem}
As authors proved in \cite{LZ00} for regular curves of Lagrangian subspaces the Robbin and Salamon index $\iRS$ for path of Lagrangian pairs defined in \cite[Section 3]{RS93} is related to the $\iCLM$-index as follows
the 
half-integer valued function given by
\begin{multline}\label{eq:RS-index-crossing}
 \iRS\big(\ell_1(t), \ell_2(t), t \in [a,b]\big)= \dfrac12 \sgn\big(\Gamma(\ell_1, \ell_2, a)\big)\\+
 \sum_{\substack{t_0 \in ]a,b[
}}
 \sgn\big(\Gamma(\ell_1, \ell_2, t_0)\big)  + \dfrac12 \sgn\big(\Gamma(\ell_1, \ell_2, b)\big).
\end{multline}
Thus, we have: 
\begin{equation}\label{eq:clm-rs}
 \iCLM(\ell_1(t),\ell_2(t);t\in[a,b])=\iRS(\ell_2(t),\ell_1(t); t \in [a,b]) - \dfrac12[h_{12}(b)-h_{12}(a)]
\end{equation}
where $h_{12}(t)\=\dim[\ell_1(t)\cap \ell_2(t)]$. We refer the interested reader to  \cite[Theorem 3.1]{LZ00} for a proof of Equation \eqref{eq:clm-rs}. 
\end{rem}

\begin{rem}
For the sake of comparison 	with the results proven in \cite{JP09} we remark that $\iCLM(L_0, \ell_2)$ can be defined by using the Seifert Van Kampen theorem for groupoids as the unique  $\Z$-valued homomorphism that it is locally defined as difference of the coindices as in  \cite[Equation (2-3)]{JP09}. It is worth noticing  that in that respect the local chart we are considering here is the opposite of the one considered in that paper.   
\end{rem}

A particular interesting situation which often occurs in the applications is the one in which $\ell(t)\=\psi(t)L$ where $\psi \in \mathscr C^1\big([a,b], \Sp(2n)\big)$. Usually, in fact, such a $\psi$ is nothing but the fundamental solution of a  linear Hamiltonian system. 

In this situation, in fact, as direct consequence of Equation \eqref{eq:forma-Q} and Equation \eqref{eq:forma-crossing}, we get that for such a path 
\begin{multline}\label{eq:forma-crossing-utile-simplettici}
 \Gamma\big(\ell(t), L_0, t_0 \big)[v]=	\langle \psi(t_0)\trasp{J_0}\psi'(t_0) v,v \rangle \qquad \forall\, v \in \psi^{-1}(t)\big(\ell(t_0)\cap L\big) \textrm{ or }\\
 \Gamma\big(\ell(t), L_0, t_0 \big)[w]=	\langle \trasp{J_0}\psi'(t_0)\psi^{-1}(t_0) w,w \rangle \qquad \forall\, w \in\ell(t_0)\cap L_0.
\end{multline}
Assuming that $\psi$ is the fundamental solution of the linear Hamiltonian system 
\begin{equation}\label{eq:hamsys-model}
	z'(t)= J_0 B(t) z(t), \qquad t \in [a,b]
\end{equation}
where $t \mapsto B(t)$ is a path of symmetric matrices, then by Equation \eqref{eq:forma-crossing-utile-simplettici}, we get that 
\begin{equation}\label{eq:plus-curve-control}
 \Gamma\big(\ell(t), L_0, t_0 \big)[w]=	\langle B(t_0) w,w \rangle \qquad \forall\, w \in\ell(t_0)\cap L_0.
\end{equation}
\begin{ex}\label{ex:utile}
In this example we compute the crossing form with respect to the Dirichlet and Neumann Lagrangian for a special curve of Lagrangian subspaces in  the symplectic space $(\R^{2n}, \omega_0)$ by using the fact that for any $L \in\Lagr(V,\omega)$, the map $\delta_L: \Sp(V, \omega) \to \Lambda(V, \omega)$ defined by $\delta_L(A)\=AL$ is a real-analytic fibration. 

Let $L_0$ be either the Dirichlet or the Neumann Lagrangian, $\ell:[a,b] \to \Lagr(n)$ be a smooth curve having a crossing instant with $\Sigma(L_0)$ at the instant $t_0 \in (a,b)$. \\ 
{\bf First case: $L_0=L_D$.\/} We assume that $\ell(t_0)$ is transverse to $L_N$ (otherwise it is enough to consider a different Lagrangian decomposition). By the local description of the atlas of the Lagrangian Grassmannian,   $\ell(t_0)$ is a graph of a (symmetric) linear map $A:\R^n \to \R^n$, namely  $\ell(t_0)=\Set{(p,q)\in \R^n \times \R^n| q= Ap}$ and hence 
\[
\ell(t_0)\cap L_D=\Set{(p,q) \in \R^n \times \R^n|q=0, \ p\in \ker A}.
\]
There exists $\varepsilon>0$ sufficiently small and $\psi: (t_0-\varepsilon, t_0+\varepsilon)\to \Sp(2n)$ with $\psi(t_0)=\Id$ such that $\ell(t)=\psi(t)\ell(t_0)$. With respect to the Lagrangian decomposition $L_D \oplus L_N = \R^{2n}$ we can write $\psi(t)$ in the block form as follows
\begin{equation}
\psi(t)\=	\begin{bmatrix}
	a(t) & b(t) \\ c(t) & d(t)	
	\end{bmatrix}.
\end{equation}
By an immediate computation, it follows that the crossing form is given by 
\begin{equation}
\Gamma(\ell, L_D, t_0)[\xi]= \langle p, \dot c(t_0) p\rangle
\end{equation}
where $p \in \ker A$ is the unique vector in $\R^n$ such that $\xi=(p,0)$.\\
{\bf Second case: $L_0=L_N$.\/} We assume that $\ell(t_0)$ is transverse to $L_D$; thus in this case, we can assume that  $\ell(t_0)=\Set{(p,q)\in \R^n \times \R^n| p= Bq}$ and hence 
\[
\ell(t_0)\cap L_N=\Set{(p,q) \in \R^n \times \R^n|p=0, \ q\in \ker B}.
\]
Under the above notation,  it follows that the crossing form is given by 
\begin{equation}
\Gamma(\ell, L_N, t_0)[\eta]= -\langle q, \dot b(t_0) q\rangle
\end{equation}
where $q \in \ker B$ is the unique vector in $\R^n$ such that $\eta=(0,q)$.

\end{ex}
\begin{rem}\label{rmk:comparison}
Before closing this section, one more comment on the Maslov  intersection index defined by author in the quoted paper.  We observe that, for a general Lagrangian path,  the (intersection) Maslov index defined by Arnol'd in \cite[Section 2]{Arn86} (namely $\iAr$) differ from  $\iCLM$ because of the contribution of the endpoints. In the aforementioned paper, author only considered paths of Lagrangian subspaces such that the starting  point  doesn't belong to the train of a  distinguished Lagrangian  $L_0$ whereas the final endpoint coincides with the vertex. However, if we restrict on this particular class of Lagrangian paths and assuming that the Hamiltonian defining these paths through the lifting to the Lagrangian Grassmannian is $L_0$-optical, then we have $
\iCLM(L_0, \ell(t); t \in [0,T])= \iAr(L_0, \ell(t), t \in [0,T])-n$ 
where $\iAr$ denotes the Maslov index defined in \cite[Section 2]{Arn86}. This fact easily follows by observing that the local contribution given by the endpoints to the $\iCLM$ index is through the coindex at the final point and the index of the starting point. 

We also observe that the Lagrangian paths defined by the evolution of  a Lagrangian subspace under the phase flow, have in general, degenerate starting point. Thus, in order to fit with the class of Lagrangian paths defined by Arnol'd  it is natural to parametrize the paths in the opposite direction. However, since the contribution at the end points is different, in the definition of $\iCLM$-index  such a re-parametrization changes the Maslov index not only for a sign changing but also for a correction term which depends upon the endpoints. This fact is pretty much put on evidence in the Sturm-type comparison theorems.
\end{rem}
We close this section by 
recalling some useful 
properties of the $\iCLM$-index. \\
\begin{itemize}
\item[]{\bf Property I (Reparametrization invariance)\/}. Let $\psi:[a,b] \to 
[c,d]$ be a 
continuous and piecewise smooth function with $\psi(a)=c$ and $\psi(b)=d$, then 
\[
 \iCLM\big(L_0, \ell(t);t \in[c,d]\big)= \iCLM(L_0, \ell(\psi(t));t \in [a,b]\big). 
\]
\item[] {\bf Property II (Homotopy invariance with respect to the ends)\/}. For 
any $s \in [0,1]$, 
let $s\mapsto \ell(s,\cdot)$ be a continuous family of Lagrangian paths 
parametrised on $[a,b]$ and 
such that $\dim\big(\ell(s,a)\cap L_0\big)$ and $\dim\big(\ell(s,b)\cap L_0\big)$ 
are constants, then 
\[
 \iCLM\big(L_0, \ell(0,t);t \in [a,b]\big)=\iCLM\big(L_0, \ell(1,t); t \in [a,b]\big).
\]
\item[]{\bf Property III (Path additivity)\/}. If $a<c<b$, then
\[
 \iCLM\big(L_0, \ell(t);t \in [a,b]\big)=\iCLM\big(L_0, \ell(t); t \in [a,c]\big)+ \iCLM\big(L_0, \ell(t); t \in [c,b]\big) 
\]
\item[]{\bf Property IV (Symplectic invariance)\/}. Let $\Phi:[a,b] \to \Sp(2n, \R)$. Then 
\[
 \iCLM\big(L_0, \ell(t);t \in [a,b]\big)= \iCLM\big(\Phi(t)L_0, \Phi(t)\ell(t); t \in [a,b]\big).
\]
\end{itemize}


\subsection{On the triple and H\"ormander index}

A crucial ingredient which somehow measure the difference of the relative Maslov index with respect to two different Lagrangian subspaces is given by the H\"ormader index. Such an index is also related to the difference of the triple index and to its interesting  generalization provided recently by the last author and his co-authors in \cite{ZWZ18}. 
For, we start with  the following definition of the H\"ormander index. 
\begin{defn}\label{def:hormander}(\cite[Definition 3.9]{ZWZ18})
Let $\lambda, \mu \in \mathscr C^0\big([a,b], \Lagr(V,\omega)\big)$ such that 
\[
\lambda(a)=\lambda_1, \quad \lambda(b)=\lambda_2 \quad  \textrm{ and } \quad \mu(a)=\mu_1, \quad \mu(b)= \mu_2.
\]
Then the {\em H\"ormander index\/} is the integer given by 
\begin{multline}
s(\lambda_1, \lambda_2; \mu_1, \mu_2)
\= 
\iCLM\big(\mu_2, \lambda(t); t \in [a,b]\big) - 
\iCLM\big(\mu_1, \lambda(t); t \in [a,b]\big) \\
=
\iCLM\big(\mu(t), \lambda_2; t \in [a,b]\big)- \iCLM\big(\mu(t), \lambda_1; t \in [a,b]\big).
\end{multline}
Compare \cite[Equation (17), pag. 736]{ZWZ18} once observing that we observe that $\iCLM(\lambda,\mu)$ corresponds to $\textrm{Mas}\{\mu,\lambda\}$ in the notation of \cite{ZWZ18}. 
\end{defn}
\paragraph{Properties of the H\"ormander index.}
We briefly recall some well-useful properties of the H\"ormander index.
\begin{itemize}
\item 	$s(\lambda_1,\lambda_2; \mu_1, \mu_2) = -s(\lambda_1,\lambda_2; \mu_2, \mu_1)$
\item $s(\lambda_1,\lambda_2; \mu_1, \mu_2) = 
-s(\mu_1, \mu_2;\lambda_1,\lambda_2) + 
\sum_{j,k \in \{1,2\}}(-1)^{j+k+1}\dim (\lambda_j \cap \mu_k)$.
\item If $\lambda_j\cap \mu_k =\{0\}$ then $s(\lambda_1,\lambda_2; \mu_1, \mu_2) = 
-s(\mu_1, \mu_2;\lambda_1,\lambda_2)$.
\end{itemize}
The H\"ormander index is computable as difference of two indices each one involving  three different Lagrangian subspaces. This index is defined in terms of the local chart representation of the atlas of the Lagrangian Grassmannian manifold, given in Equation \eqref{eq:carte}. 

\begin{defn}\label{def:kashi}
Let $\alpha,\beta,\gamma \in \Lagr(V,\omega)$, $\epsilon\= \alpha \cap \beta + \beta \cap \gamma$ and let $\pi\=\pi_\epsilon$ be the projection in the symplectic reduction of $V$ mod $\epsilon$.   
  We term {\em triple index\/} the integer defined by
\begin{multline}\label{eq:triple}
\itriple(\alpha, \beta, \gamma)\= \coindex Q(\pi \alpha, \pi \beta; \pi \gamma)	+\dim (\alpha \cap \gamma) -\dim (\alpha\cap \beta \cap \gamma)\\
\leq n-\dim (\alpha \cap \beta)-\dim( \beta \cap \gamma) + \dim (\alpha \cap \beta \cap \gamma).
\end{multline}
\end{defn}
\begin{rem}\label{rem:molto-utile-stima}
Definition \ref{def:kashi} is well-posed and we refer the interested reader to \cite[Lemma 2.4]{Dui76} and  \cite[Corollary 3.12 \& Lemma 3.13]{ZWZ18} for further details).  It is worth noticing that $Q(\pi \alpha, \pi \beta; \pi \gamma)$ is a quadratic form on $\pi\alpha$. Being the reduced space $V_\epsilon$  a $2(n-\dim \epsilon)$ dimensional subspace, it follows that inertial indices of  $Q(\pi \alpha, \pi \beta; \pi \gamma)$ are integers between $\{0, \ldots,n-\dim \epsilon\}$.
\end{rem}
\begin{rem}
		It is worth noticing that	for arbitrary Lagrangian subspaces $\alpha,\beta,\gamma$ , $Q(\alpha,\beta,\gamma)$ is well-defined and it is a quadratic form on $\alpha\cap(\beta+\gamma)$.
	Furthermore, we have $n_+Q(\alpha,\beta,\gamma)=n_+ Q(\pi\alpha,\pi\beta,\pi\gamma)$.  So we can also define the triple index as 
	\[
	\itriple(\alpha, \beta, \gamma)\= \coindex Q( \alpha,  \beta;  \gamma)	+\dim (\alpha \cap \gamma) -\dim (\alpha\cap \beta \cap \gamma).
	\]
Authors in \cite[Lemma 3.2]{ZWZ18}  give a useful property for calculating such a quadratic form.
\begin{equation}\label{eq:invariance_Q}
\coindex Q(\alpha,\beta,\gamma)=\coindex Q(\beta,\gamma,\alpha)= \coindex Q(\gamma,\alpha,\beta).
\end{equation}
\end{rem}
We observe that if $(\alpha,\beta)$ is a Lagrangian decomposition of $(V,\omega)$ and $\beta \cap \gamma=\{0\}$ then $\pi$ reduces to the identity and both  terms $\dim (\alpha \cap \gamma)$ and $\dim (\alpha\cap \beta \cap \gamma)$ drop down. In this way the triple index is nothing different from the the quadratic form $Q$ defining the local chart of the atlas of $\Lagr(V,\omega)$ given in Equation \eqref{eq:carte}.  It is possible to prove (cfr. \cite[proof of the Lemma 3.13]{ZWZ18}) that 
\begin{equation}\label{eq:kernel-dim-q-form}
	\dim(\alpha \cap \gamma) -\dim(\alpha \cap \beta \cap \gamma)= \nullity Q(\pi \alpha, \pi \beta; \pi \gamma),
\end{equation}
where we denoted by $\nullity Q$ the nullity (namely the kernel dimension of the quadratic form $Q$).
By summing up Equation \eqref{eq:triple} and Equation \eqref{eq:kernel-dim-q-form}, we finally get 
\begin{equation}\label{eq:triple-coindex-extended}
\itriple(\alpha, \beta, \gamma)= \noo{+}Q(\pi \alpha, \pi \beta; \pi \gamma)
\end{equation}
 where $\noo{+} Q$ denotes the so-called {\em extended coindex\/} or {\em generalized coindex\/} (namely the coindex plus the nullity) of the quadratic form $Q$. (Cfr. \cite[Lemma 2.4]{Dui76} for further details).
\begin{lem}\label{thm:properties}
Let $\lambda \in \mathscr C^1\big([a,b], \Lagr(V,\omega)\big)$. Then, for every $\mu \in \Lagr(V, \omega)$, we have 
\begin{enumerate}
	\item[\textrm{ \bf{(I)} }] $s\big(\lambda(a), \lambda(b); \lambda(a), \mu \big)= -\itriple\big(\lambda(b), \lambda(a),\mu\big)\leq 0$, 
	\item[\textrm{ \bf{(II)} }] $s\big(\lambda(a), \lambda(b); \lambda(b), \mu \big)= \itriple\big(\lambda(a), \lambda(b),\mu\big)\geq 0$.
\end{enumerate}
\end{lem}
\begin{proof}
	For the proof, we refer the interested reader to \cite[Corollary 3.16]{ZWZ18}.
\end{proof}
The next result, which is the main result of \cite{ZWZ18}, allows to reduce the computation of the H\"ormander index to the computation of the triple index. 
\begin{prop}{\bf \cite[Theorem 1.1]{ZWZ18}\/}\label{thm:mainli} 
Let $(V,\omega)$ be a $2n$-dimensional symplectic space and let  $\lambda_1, \lambda_2, \mu_1, \mu_2 \in \Lagr(V,\omega)$. Under the above notation, we get 
\begin{equation}
s(\lambda_1, \lambda_2,\mu_1,\mu_2)=\itriple(\lambda_1,\lambda_2,\mu_2)- 	\itriple(\lambda_1,\lambda_2,\mu_1)\\
=\itriple(\lambda_1,\mu_1,\mu_2)- \itriple(\lambda_2,\mu_1,\mu_2)
\end{equation}
\end{prop}
\begin{rem}
We emphasize that no transversality conditions are assumed on the four Lagrangian subspaces in Proposition \ref{thm:mainli} 
\end{rem}


\section{On the Spectral Flow }\label{sec:spectral-flow}
%
Let $\mathcal W, \mathcal H$ be  real separable Hilbert spaces with a dense 
and  continuous inclusion $\mathcal W \hookrightarrow \mathcal H$. In what follows we use the following notation.  $\mathcal{B}(\mathcal W,\mathcal H)$ denotes the Banach  space of all linear bounded 
operators; $\mathcal{B}^{sa}(\mathcal W, \mathcal H)$ denotes the set of all  linear bounded selfadjoint operators when regarded as operators  on  $\mathcal H$. $\mathcal{BF}^{sa}(\mathcal W, \mathcal H)$  denotes the set of all linear and  bounded selfadjoint Fredholm operators.
Let now $T \in \mathcal{BF}^{sa}(\mathcal W,\mathcal H)$, then either $0$ is  not  in $\sigma(T)$ or it is in $\sigma_{disc}(T)$ and, as a consequence of the Spectral Decomposition Theorem (cf. \cite[Theorem  6.17, Chapter  III]{Kat80}), the following orthogonal decomposition holds $ \mathcal W = E_-(T) \oplus \ker T \oplus E_+(T),$
with the property
\[
 \sigma(T) \cap(-\infty, 0)= \sigma\left(T_{E_-(T)}\right) \textrm{ and } 
 \sigma(T) \cap(0,+\infty)= \sigma\left(T_{E_+(T)}\right).
\]
\noindent
\begin{defn}\label{def:Morseindex}
Let $T \in \mathcal{BF}^{sa}(\mathcal W,\mathcal H)$. If $\dim E_-(T)<\infty$ 
(resp.  $\dim 
E_+(T)<\infty$), 
we define its {\em Morse index\/} (resp. {\em Morse co-index\/})
as the integer denoted by $\iMor(T)$  (resp. $\coiMor(T)$) and defined as $ \iMor(T) \= \dim E_-(T)\qquad \big(\textrm{resp. } \coiMor(T)\= \dim E_+(T)\big).$
\end{defn}
We are now in position to introduce the spectral flow. 
Given a  $\mathscr C^1$-path  $L:[a,b]\to\mathcal{BF}^{sa}(\mathcal W, \mathcal 
H)$, the spectral flow of $L$ counts the net number of eigenvalues crossing 0. 
\begin{defn}\label{def:crossing}
An instant $t_0 \in (a,b)$ is called a \emph{crossing instant} (or {\em 
crossing\/} for short) if $\ker  L_{t_0} \neq \{0\}$. The \emph{crossing form} 
at a crossing $t_0$ is the quadratic form defined by 
\[
 \Gamma( L, t_0): \ker  L_{t_0} \to \R, \ \ \Gamma( L, 
t_0)[u] \=\langle 
 \dot{ L}_{t_0} u, u \rangle_{\mathcal H},
\]
where we denoted by $\dot{L}_{t_0}$ the derivative of $L$ 
with respect to the parameter $t \in [a,b]$ at the point $t_0$.
A crossing is called \emph{regular}, if $\Gamma( L, t_0)$ is 
non-degenerate. If $t_0$ is a crossing instant for $L$, we refer to 
$m(t_0)$ the dimension of $\ker  L_{t_0}$.
\end{defn}

\begin{rem}
It is worth noticing that regular crossings are isolated, and hence, on a compact interval are in a finite number. 
\end{rem}
In the case of regular curve (namely a curve having only regular crossings)  we introduce the following Definition. 
\begin{defn}\label{def:new-spectralflow-def}
 Let  $L:[a,b]\to\mathcal{BF}^{sa}(\mathcal  H)$ be a $\mathscr 
C^1$-path and 
 we assume that it has only regular crossings. Then 
 \begin{equation}\label{eq:spectral-flow-crossings}
\spfl(L; [a,b])= \sum_{t \in (a,b)} \sgn \Gamma(L, t)- 
\iMor\big(\Gamma(L,a)\big)
+ \coiMor\big(\Gamma(L,b)\big),
\end{equation}
where the sum runs over all regular (and hence in a finite number) strictly contained in  $[a,b]$.
\end{defn}
We recall the following well-known result.
\begin{lem}\label{thm:perturbazione}
There exists $\varepsilon > 0$ such that
\begin{itemize}
\item $A + \delta \Id$ is a path in $\mathcal{BF}^{sa}(\mathcal W, \mathcal  H)$ for all $|\delta| \leq \varepsilon$;
\item $A + \delta \Id$ has only regular crossings for almost every $\delta \in (-\varepsilon, \varepsilon)$.
\end{itemize}
\end{lem}
\begin{defn}\label{def:positive-paths}
	The $\mathscr C^1$-path $L:[a,b]\ni t \mapsto L_t\in \mathcal{BF}^{sa}(\mathcal  H)$ is termed {\em positive\/} or {\em plus\/} path, if at each crossing instant $t_*$ the crossing form $\Gamma(L, t_*)$ is positive definite.  
\end{defn}
\begin{rem}
We observe that in the case of a positive path, each crossing is regular and in particular the total number of crossing instants on a compact interval is finite. Moreover the local contribution at each crossing to the spectral flow is given by the dimension of the intersection. Thus given a positive path $L$, the spectral flow is given by 
 \[
\spfl(L; [a,b])= \sum_{t \in (a,b)} \dim  \ker L(t)+ \dim  \ker L(b). 
\]
\end{rem}
\begin{defn}\label{def:admissible-paths-operators}
The path $L:[a,b]\to\mathcal{BF}^{sa}(\mathcal  H)$ is termed {\em admissible\/} provided it has invertible endpoints. 
\end{defn}
For paths of bounded self-adjoint Fredholm operators parametrized on $[a,b]$ which are compact perturbation of a fixed operator,  the spectral flow given in Definition \ref{def:new-spectralflow-def}, can be characterized as the relative Morse index of its endpoints. More precisely, the following result holds. 
\begin{prop}\label{thm:spfl-operatori-diff-rel-morse}
Let us consider the  path  $L: [a,b] \to\mathcal{BF}^{sa}(\mathcal  H)$  and we assume that for every $t \in [a,b]$, the operator  $L_t- L_a$ is compact.   Then
\begin{equation}\label{eq:equality-spfl-relmorse}
-\spfl(L; [a,b])=\irel(L_a, L_b).
\end{equation}
Moreover if $L_a$ is essentially positive, then we have 
\begin{equation}\label{eq:diff-Morse}
-\spfl(L; [a,b])
=\iMor (L_b)-\iMor(L_a)
\end{equation}
and if furthermore  $L_b$ is positive definite, then 
\begin{equation}
\spfl(L; [a,b])=\iMor(L_a).
\end{equation}
\end{prop}
\begin{proof}
The proof of the equality in  Equation~\eqref{eq:equality-spfl-relmorse} is an immediate consequence of the fixed end homotopy properties of the spectral flow. For, let $\varepsilon >0 $ and let us consider the two-parameter family 
$L:[0,1]\times [a,b] \to \mathcal{BF}^{sa}(\mathcal  H) \textrm{ defined by } L(s,t)\= L_t+ s \,\varepsilon\, \Id. $
By the homotopy property of the spectral flow, we get that 
\begin{multline}\label{eq:fff}
\spfl(L_t; t \in [a,b])\\	=  \spfl(L_a+s \varepsilon \Id, s \in [0,1]) + \spfl(L_t+\varepsilon\Id, t \in [a,b])- \spfl(L_b + s \varepsilon\Id, s \in [0,1])\\=  \spfl(L_t+\varepsilon\Id, t \in [a,b])
\end{multline}
where the last equality in  Equation~\eqref{eq:fff} is consequence if the positivity of all the involved paths.  By choosing a maybe smaller $\varepsilon>0$ the path  $t\mapsto L_t+\varepsilon\Id$ is admissible (in the sense of Definition  \ref{def:admissible-paths-operators}). The conclusion, now readily follows by applying  \cite[Proposition 3.3]{FPR99} (the minus sign appearing is due to a different choosing convention for the spectral flow. 

In order to prove the second claim, it is enough to observe that if $L_a$ is essentially positive, then $L$ is a  path entirely contained in the (path-connected component) $\mathcal{BF}_+^{sa}(\mathcal  H)$. The proof of the equality in Equation~\eqref{eq:diff-Morse} is now a direct  consequence of Equation the previous argument and \cite[Proposition 3.9]{FPR99}. The last can be deduced by Equation~\eqref{eq:diff-Morse} once observed that $\iMor(L_b)=0$. This concludes the proof. 
\end{proof}
\begin{rem}
	We observe that  a direct proof of Equation~\eqref{eq:diff-Morse} can be easily conceived as direct consequence of the homotopy properties of $\mathcal{BF}_+^{sa}(\mathcal  H)$.
\end{rem}

\begin{rem}
 We observe that the definition of spectral flow for bounded selfadjoint Fredholm operators given in Definition \ref{def:new-spectralflow-def} is slightly different from the standard definition given in literature in which only continuity is required on the regularity of the path.  (For further details, we refer the interested reader to 
\cite{RS95,Wat15} and 
 references therein). Actually Definition \ref{def:new-spectralflow-def}  represents an efficient way for 
 computing the spectral flow even if it requires more regularity as well as a  transversality assumption 
 (the regularity of each crossing instant). However, it is worth to mentioning that, the spectral flow is a fixed endpoints homotopy invariant and for admissible paths (meaning for paths having invertible endpoints) is a free homotopy invariant. By density arguments, we observe that a $\mathscr C^1$-path 
 always exists  in any fixed endpoints homotopy class of the original path.  
\end{rem}

\begin{rem} 
 It is worth noting, as already observed by author in \cite{Wat15}, that the spectral flow can be 
 defined in the more general case of continuous 
 paths of closed unbounded selfadjoint Fredholm operators that are 
 continuous with respect to the (metric) gap-topology. However in the special case in 
 which the domain of the operators is fixed, then the closed path of unbounded 
 selfadjoint Fredholm operators can be regarded as a continuous path 
 in $\mathcal{BF}^{sa}(\mathcal W, \mathcal  H)$. Moreover  this path is also continuous 
 with respect to the aforementioned gap-metric topology.
 
 The advantage to regard the paths in  $\mathcal{BF}^{sa}(\mathcal W, \mathcal  H)$ is that the 
 theory is straightforward as in the bounded case and, clearly, it is sufficient for the applications  
 studied in the present manuscript. 
\end{rem}



\vspace{2cm}
	\noindent
	\textsc{Prof. Vivina L.~Barutello}\\
	Dipartimento di Matematica \lq\lq G.~Peano\rq\rq\\
	Università degli Studi di Torino\\
	Via Carlo Alberto, 10 \\
	10123 Torino \\
	Italy\\
	E-mail: \email{vivina.barutello@unito.it}   
   
\vspace{1cm}
\noindent
\textsc{Prof. Daniel C. Offin}\\
Department of Mathematics and Statistics\\
Queens University\\
Kingston, Ontario, Canada, KZL 3N6\\
E-mail: \email{offind@mast.queensu.ca}

\vspace{1cm}
\noindent
\textsc{Prof. Alessandro Portaluri}\\
DISAFA\\
Università degli Studi di Torino\\
Largo Paolo Braccini 2 \\
10095 Grugliasco, Torino\\
Italy\\
Website: \url{https://sites.google.com/view/alessandro-portaluri/home}\\
E-mail: \email{alessandro.portaluri@unito.it}

\vspace{1cm}
\noindent
\textsc{Prof. Li Wu}\\
Department of Mathematics\\
Shandong University\\
Jinan,Shandong, 250100\\
The People's Republic of China \\
China\\
E-mail: \email{vvvli@sdu.edu.cn}


\begin{thebibliography}{99}

\bibitem[Arn67]{Arn67}
{\sc Arnol'd, V. I.} 
\newblock On a characteristic  class entering into conditions of quantization.
\newblock (Russian) Funkcional. Anal. i Priložen. 1 1967 1--14.

\bibitem[Arn86]{Arn86}
{\sc Arnol'd, V. I.} 
\newblock  Arnol'd, V. I. Sturm theorems and symplectic geometry.
\newblock (Russian) Funktsional. Anal. i Prilozhen. 19 (1985), no. 4, 1--10, 95.

\bibitem[APS08]{APS08}
{\sc  Abbondandolo, Alberto; Portaluri, Alessandro; Schwarz, Matthias}
\newblock The homology of path spaces and Floer homology with conormal boundary 
conditions. 
\newblock J. Fixed Point Theory Appl. 4 (2008), no. 2, 263--293.

  
   \bibitem[BJP14]{BJP14}
  {\sc Vivina Barutello, Riccardo D. Jadanza, Alessandro Portaluri}
  \newblock Morse index and linear stability of the Lagrangian circular
  orbit in a three-body-type problem via index theory.
  \newblock Preprint available on \texttt{http://arxiv.org/abs/1406.3519}
  
    \bibitem[BJP16]{BJP16}
  {\sc Barutello, Vivina; Jadanza, Riccardo D.; Portaluri, Alessandro} 
  \newblock Morse index and linear stability of the Lagrangian circular orbit in a three-body-type problem via index theory. 
  \newblock Arch. Ration. Mech. Anal. 219 (2016), no. 1, 387--444.
  

   \bibitem[BP92]{BP92}
  {\sc Bialy, M.; Polterovich, L.} 
  \newblock  Hamiltonian diffeomorphisms and Lagrangian distributions.   \newblock Geom. Funct. Anal. 2 (1992), no. 2, 173--210.

\bibitem[CLM94]{CLM94}
{\sc Cappell, Sylvain E.; Lee, Ronnie; Miller, Edward Y.}
\newblock On the Maslov index. 
\newblock Comm. Pure Appl. Math. 47 (1994), no. 2, 121--186.

\bibitem[CGIP03]{CGIP03}
{\sc Contreras,  Gonzalo; Gambaudo, Jean-Marc; Iturriaga, Renato; Paternain, Gabriel P.}
\newblock The asymptotic Masloc index and its applications. 
\newblock Ergod. Th. \& Dynam. Sys.  (2003), no. 23, 1415--1443.

\bibitem[DDP08]{DDP08}
{\sc De Gosson, Maurice; De Gosson, Serge; Piccione, Paolo}
\newblock On a product formula for the Conley-Zehnder index of symplectic paths and its applications. 
\newblock Ann. Global Anal. Geom. 34 (2008), no. 2, 167--183. 

\bibitem[Dui76]{Dui76}
{\sc Duistermaat, J. J.}
\newblock On the Morse index in variational calculus. 
\newblock Advances in Math. 21 (1976), no. 2, 173--195. 

\bibitem[FK14]{FK14}
{\sc Frauenfelder, Urs; van Koert, Otto}
\newblock The Hörmander index of symmetric periodic orbits.
\newblock Geom. Dedicata 168 (2014), 197--205.


\bibitem[FPR99]{FPR99}
{\sc Fitzpatrick, Patrick M.; Pejsachowicz, Jacobo; Recht, Lazaro}
\newblock Spectral flow and bifurcation of critical points of strongly-indefinite functionals. I. General theory. 
\newblock J. Funct. Anal. 162 (1999), no. 1, 52--95. 

\bibitem[GPP04]{GPP04}
{\sc Giambò, Roberto; Piccione, Paolo; Portaluri, Alessandro}
\newblock Computation of the Maslov index and the spectral flow via partial 
signatures 
\newblock C. R. Math. Acad. Sci. Paris 338 (2004), no. 5, 397--402.

\bibitem[HS09]{HS09}
{\sc Hu, Xijun; Sun, Shanzhong}
\newblock Index and Stability of Symmetric Periodic Orbits in Hamiltonian Systems with Application to  Figure-Eight Orbit, \newblock Comm. Math. Phys. 290 (2009), 737--777. 

\bibitem [HWY18]{HWY18}
{\sc Hu, Xijun, Wu, Li,  Yang, Ran}
\newblock  Morse Index Theorem of Lagrangian Systems and Stability of Brake Orbit, \newblock Journal of Dynamics and Differential Equations. https://doi.org/10.1007/s10884-018-9711-x
\bibitem[Kat80]{Kat80}
{\sc Tosio Kato}
\newblock Perturbation Theory for linear operators.
\newblock  Grundlehren der Mathematischen Wissenschaften, 132,
Springer-Verlag (1980).

\bibitem[KOP19]{KOP19}
{\sc Kavle, Henry; Offin, Daniel; Portaluri, Alessandro}
\newblock Keplerian orbits through the Conley-Zehnder index
\newblock Preprint available at \url{https://arxiv.org/pdf/1908.00075.pdf}

\bibitem[JP09]{JP09}
{\sc Javaloyes, Miguel Ángel; Piccione, Paolo}
\newblock Comparison results for conjugate and focal points in semi-Riemannian 
geometry via Maslov index.
\newblock Pacific J. Math. 243 (2009), no. 1, 43--56. 




\bibitem[LV80]{LV80} 
{\sc Lion, Gérard; Vergne, Michèle}
\newblock The Weil representation, Maslov index and theta series. 
\newblock Progress in Mathematics, 6. Birkh\"auser, Boston, Mass., 1980.

\bibitem[Lon02]{Lon02}
{\sc  Long, Yiming}
\newblock Index theory for symplectic paths with applications. 
\newblock Progress in Mathematics, 207. Birkh\"auser Verlag, Basel, 2002.

\bibitem[LZ00]{LZ00}
{\sc  Long, Yiming; Zhu, Chaofeng}
\newblock Maslov-type index theory for symplectic paths and spectral flow. I. 
\newblock Chinese Ann. Math. Ser. E 21 (2000), no. 4, 89--108.




\bibitem[Mon18]{Mon18}
{\sc Montgomery, Richard}
\newblock Metric cones, N-body collisions, and Marchal's lemma
\newblock Preprint available at \url{https://arxiv.org/pdf/1804.03059.pdf}


  \bibitem[MPP05]{MPP05}
  {\sc Musso, Monica; Pejsachowicz, Jacobo; Portaluri, Alessandro,}
  \newblock A Morse index theorem for perturbed geodesics on semi-Riemannian 
  manifolds. 
  \newblock Topol. Methods Nonlinear Anal. 25 (2005), no. 1, 69--99.
  
  
  \bibitem[MPP07]{MPP07}
  {\sc Musso, Monica; Pejsachowicz, Jacobo; Portaluri, Alessandro}
  \newblock Morse index and bifurcation of $p$-geodesics on 
  semi Riemannian manifolds. 
  \newblock ESAIM Control Optim. Calc. Var. 13 (2007), no. 3, 598--621.
  
\bibitem[Off00]{Off00}
{\sc Offin, Daniel}
\newblock Hyperbolic minimizing geodesics. 
\newblock Trans. Amer. Math. Soc. 352 (2000), no. 7, 3323--3338. 

\bibitem[PPT04]{PPT04}
{\sc Piccione, Paolo; Portaluri, Alessandro; Tausk, Daniel V.}
\newblock Spectral flow, Maslov index and bifurcation of semi-Riemannian 
geodesics. 
\newblock Ann. Global Anal. Geom. 25 (2004), no. 2, 121--149. 



\bibitem[Pin75]{Pin75}
{\sc Pin, Ong Chong}
\newblock Curvature and mechanics. 
\newblock Advances in Mathematics 15 (1975), 269--311.


\bibitem[Por08]{Por08}
{\sc Portaluri, Alessandro} 
\newblock Maslov index for Hamiltonian systems. 
\newblock Electron. J. Differential Equations 2008, No. 09.

\bibitem[PWY19]{PWY19}
{\sc Portaluri, Alessandro; Wu, Li; Yang Ran}
\newblock Linear instability for periodic orbits of non-autonomous Lagrangian systems
\newblock Preprint available at \url{https://arxiv.org/pdf/1907.05864.pdf}

\bibitem[RS93]{RS93}
{\sc Robbin, Joel; Salamon, Dietmar} 
\newblock The Maslov index for paths. 
\newblock Topology 32 (1993), no. 4, 827--844. 
  
  \bibitem[RS95]{RS95}
  {\sc Robbin, Joel; Salamon, Dietmar}
  \newblock The spectral flow and the Maslov index. 
  \newblock Bull. London Math. Soc. 27 (1995), no. 1, 1--33.
%

\bibitem[Sma70a]{Sma70a}
{\sc Smale, Steve} 
\newblock Topology and mechanics. I. 
\newblock Invent. Math. 10 (1970), 305--331. 

\bibitem[Sma70b]{Sma70b}
{\sc Smale, Steve}
\newblock  Topology and mechanics. II. The planar n-body problem.
\newblock Invent. Math. 11 (1970), 45--64.

\bibitem[Wat15]{Wat15}
{\sc Waterstraat, Nils}
\newblock Spectral flow, crossing forms and homoclinics of Hamiltonian systems. 
\newblock Proc. Lond. Math. Soc. (3) 111 (2015), no. 2, 275--304.


\bibitem[ZWZ18]{ZWZ18}
{\sc Zhou, Yuting; Wu, Li; Zhu, Chaofeng}
\newblock H\"ormander index in finite-dimensional case. 
\newblock Front. Math. China 13 (2018), no. 3, 725--761.

\end{thebibliography}
\end{document}